\documentclass[11pt]{amsart}

\usepackage{amsmath, amssymb, enumerate, amsthm, setspace, tikz-cd, hyperref}
\usepackage[shortcuts]{extdash}
\usepackage[all]{xy}
\theoremstyle{plain}
\newtheorem{theorem}{Theorem}[section]

\newtheorem*{corollary*}{Corollary}
\newtheorem{lemma}[theorem]{Lemma}
\newtheorem{proposition}[theorem]{Proposition}
\newtheorem{corollary}[theorem]{Corollary}

\theoremstyle{definition}
\newtheorem{definition}[theorem]{Definition}
\newtheorem{exampletemp}[theorem]{Example}
\newtheorem{remark}[theorem]{Remark}

\DeclareMathOperator{\Aut}{Aut}

\DeclareMathOperator{\Ad}{Ad}

\DeclareMathOperator{\id}{id}
\DeclareMathOperator{\Ext}{Ext}

\DeclareMathOperator{\conv}{conv}

\newcommand{\cstar}{\ensuremath{\mathrm{C}^{*}}}

\begin{document}

\title[Crossed Products of Nuclear C$^*$-Algebras and Their Traces]{Crossed Products of Nuclear C$^*$-Algebras by Free Groups and Their Traces}
\author[T.\ Rainone and C.\ Schafhauser]{Timothy Rainone and Christopher Schafhauser}
\date{\today}
\address{Department of Pure Mathematics, University of Waterloo, 200 University Avenue West, Waterloo, ON, Canada, N2L 3G1}
\email{trainone@uwaterloo.ca}
\email{cschafhauser@uwaterloo.ca}
\date{\today}

\begin{abstract}
  We study the matricial field (MF) property for certain reduced crossed product \cstar-algebras and their traces. Using classification techniques and induced $\mathrm{K}$-theoretic dynamics, we show that reduced crossed products of ASH-algebras of real rank zero by free groups are MF if and only if they are stably finite. We also examine traces on these crossed products and show they always admit certain finite dimensional approximation properties.  Combining these results with recent progress in Elliott's Classification Program shows that if $A$ is a separable, simple, unital, nuclear, monotracial \cstar-algebra satisfying the UCT, then $A \rtimes_{\lambda} \mathbb{F}_r$ is MF for any action of $\mathbb{F}_r$ on $A$.  By appealing to a result of Ozawa, R{\o}rdam, and Sato, we show that discrete groups of the form $G \rtimes \mathbb{F}_r$ with $G$ amenable admit MF reduced group C$^*$-algebras.
\end{abstract}

\maketitle

\section{Introduction}

Matricial field \cstar-algebras were introduced by Blackadar and Kirchberg in~\cite{BlackadarKirchberg}. These are stably finite \cstar-algebras which admit external finite dimensional approximations capturing the linear, multiplicative, and metric properties of the algebra. More precisely, a separable C$^*$-algebra $A$ is called \emph{matricial field (MF)} if there is an embedding \mbox{$A \hookrightarrow \mathcal{Q}_\omega$} where $\mathcal{Q}_\omega$ denotes the norm ultrapower of the universal UHF-algebra $\mathcal{Q}$ with respect to some free ultrafilter $\omega$ on $\mathbb{N}$. In the nuclear setting, the class of MF \cstar-algebras coincides with the class of quasidiagonal (QD) \cstar-algebras as a well known consequence of the Choi-Effros Lifting Theorem and Voiculescu's noncommutative Weyl-von Neumann Theorem.  The MF property is the \cstar-analogue of admitting tracial microstates; i.e. embeddability into the tracial ultrapower of the hyperfinite II$_1$ factor.

Of central importance in this paper is the following question asked by Blackadar and Kirchberg in \cite{BlackadarKirchberg}: is every stably finite C$^*$-algebra MF? This is the \cstar-version of Connes's Embedding Problem. Restricting to the nuclear case, they also ask if every stably finite, nuclear C$^*$-algebra is quasidiagonal.  These questions have attracted considerable attention and interest among C$^*$-algebraists.  This is especially true in the nuclear case where it is known that quasidiagonality is a necessary condition in the stably finite version of Elliott's Classification Program.

The Blackadar-Kirchberg questions have also been of interest in the theory of dynamical systems.  In particular, if $X$ is a compact metric space and $G$ is a countable, discrete group acting on $X$, when is $C(X) \rtimes_\lambda G$ QD or MF?  In the classical case where $G = \mathbb{Z}$, a complete solution was obtained by Pimsner in \cite{Pimsner}.  This was extended to free groups by Kerr and Nowak in \cite{KerrNowak} under the assumption that either $X$ is zero-dimensional or the action is minimal.  Both results have analogues for actions on AF-algebras (see \cite{BrownAFE} and \cite{Rainone}).  In this piece, we extend this characterization from AF-algebras to approximately subhomogeneous (ASH) algebras of real rank zero (Theorem 1.1 below).

In \cite{BrownQDTraces}, N.\ Brown introduced and developed tracial versions of these conditions.  In particular, a trace $\tau$ on a C$^*$-algebra $A$ is MF if there is a trace-preserving $\ast$-homomorphism $A \rightarrow \mathcal{Q}_\omega$ and is QD if, in addition, the map can be chosen to have a completely positive, contractive lifting.\footnote{The MF version is not mentioned in \cite{BrownQDTraces}, but it is an exact analogue of the QD traces and many of the basic properties still hold.} The norm and tracial versions of the QD and MF properties are closely related.  For instance, it is clear that every unital, MF (resp. QD) C$^*$-algebra admits an MF trace (resp. a QD trace) since $\mathcal{Q}_\omega$ admits a trace.  Conversely, if $A$ admits a faithful, MF (resp. QD) trace, then $A$ is MF (resp. QD).  In particular, a simple, unital C$^*$-algebra is MF (resp. QD) if and only if it admits an MF (resp. QD) trace.

A recent result of Tikuisis, White, and Winter (see \cite{TikuisisWhiteWinter}) states that if $A$ is a separable, nuclear C$^*$-algebra satisfying the Universal Coefficient Theorem (UCT), then every faithful trace on $A$ is quasidiagonal.  In particular, separable, simple, unital, nuclear, stably finite C$^*$-algebras satisfying the UCT are quasidiagonal.  Also, for every discrete, amenable group $G$, the reduced group C$^*$-algebra C$^*_\lambda(G)$ is quasidiagonal which confirms of conjecture of Rosenberg; the converse is also true by a well known result of Rosenberg (see the appendix to \cite{Rosenberg}).  An extension of this result is given in Theorem 1.4 below.

The theme throughout our work is that of K-theoretic dynamics. Given a \cstar-dynamical system, there is an induced action on the K$_0$-group of the algebra, and, in the presence of sufficiently many projections, studying this action yields meaningful information about the structure of the reduced crossed product. Both stable finiteness and the MF property have K-theoretic expressions that turn out to be equivalent (Theorem \ref{thm:K0QDQ}). By borrowing lifting and uniqueness results from the classification literature, we answer the C$^*$-Connes Embedding Problem in the affirmative for certain reduced crossed products of nuclear C$^*$-algebras by free groups. The following is our main result.

\begin{theorem}\label{thm:MainResult}
Let $A$ be an ASH-algebra of real rank zero.  Given an action $\alpha : \mathbb{F}_r \curvearrowright A$ of a free group on $A$, the following are equivalent:
\begin{enumerate}
  \item $A \rtimes_\lambda \mathbb{F}_r$ is MF;
  \item $A \rtimes_\lambda \mathbb{F}_r$ is stably finite;
  \item the subgroup of $\mathrm{K}_0(A)$ generated by $\{ x - \alpha_s(x) : x \in \mathrm{K}_0(A), s \in \mathbb{F}_r \}$ contains no non-zero positive elements of $\mathrm{K}_0(A)$.
\end{enumerate}
Moreover, every trace on $A \rtimes_\lambda \mathbb{F}_r$ is MF.
\end{theorem}

Theorem \ref{thm:MainResult} is of interest even when the acting group is $\mathbb{Z}$ (i.e.\ when $r =1$).  In this case, (1) may be replaced with ``$A \rtimes \mathbb{Z}$ is quasidiagonal'' since $A \rtimes \mathbb{Z}$ is nuclear.  In both Pimsner's and Brown's results mentioned above, a stronger version of quasidiagonality is obtained; namely, the crossed product embeds into an AF-algebra.  It would be interesting to know if the conditions above are also equivalent to the crossed product embedding into an AF-algebra when the acting group is $\mathbb{Z}$.  Our techniques break down in this setting since AF-embeddability is not known to be a local condition; however, the existence of an AF-embedding is conjectured to be equivalent to quasidiagonality and exactness.

It is not hard to see the real rank zero condition is necessary to get the K-theoretic characterizations of the MF property and of stably finiteness given above since if there are very few projections in $A$, the $\mathrm{K}_0$-group may not contain much information about the algebra or the dynamics.  For a concrete example, if $A = C(\mathbb{T})$, then $\mathrm{K}_0(A) = \mathbb{Z}$ and there is a unique action of $\mathbb{F}_r$ on $\mathrm{K}_0(A)$.

Combining Theorem 1.1 with recent breakthroughs in the classification program also gives structural information about free group actions on many simple, nuclear C$^*$-algebras.

\begin{theorem}\label{thm:ByTheMegaTheorem...}
Let $A$ be a separable, simple, unital C*-algebra which satisfies the UCT.  Assume $A \otimes \mathcal{Q}$ has finite nuclear dimension and $\mathrm{K}_0(A)$ separates traces on $A$.  Given an action $\alpha: \mathbb{F}_r \curvearrowright A$, the following are equivalent:
\begin{enumerate}
  \item $A \rtimes_{\lambda} \mathbb{F}_r$ is MF;
  \item $A \rtimes_{\lambda} \mathbb{F}_r$ is stably finite;
  \item $A$ admits an invariant trace.
\end{enumerate}
Moreover, every trace on $A \rtimes_{\lambda} \mathbb{F}_r$ is MF.
\end{theorem}

The assumption that $A \otimes \mathcal{Q}$ has finite nuclear dimension in Theorem \ref{thm:ByTheMegaTheorem...} may hold automatically when $A$ is nuclear.  By the main result of \cite{BBSTWW}, this holds whenever the set of extreme points of $T(A)$, the simplex of tracial states on $A$, is closed in the weak* topology.  In particular, this is the case when $A$ has a unique trace as was originally shown in \cite{SatoWhiteWinter} extending the work of \cite{MatuiSato}.  Also, combining this result with an observation of N. Brown yields a large class of C$^*$-algebras for which the Brown-Douglas-Fillmore semigroup $\Ext$ is not a group.  The result below is a significant generalization of Corollary 3.14 in \cite{Rainone}.

\begin{corollary}\label{cor:Monotracial}
Let $A$ be a separable, simple, unital, nuclear, monotracial C*-algebra satisfying the UCT and let $\alpha$ be an action of $\mathbb{F}_r$ on $A$.
\begin{enumerate}
\item $A \rtimes_\lambda \mathbb{F}_r$ is MF.
\item If $r \geq 2$, then $\Ext(A \rtimes_\lambda \mathbb{F}_r)$ is not a group.
\end{enumerate}
\end{corollary}

Finally, combining the results above with the work of Ozawa, R{\o}rdam, and Sato on amenable groups, a new class of groups with MF reduced C$^*$-algebras is obtained.  In particular, we have the following theorem.

\begin{theorem}\label{thm:AmenableByFreeGroups}
Let $G$ be an amenable group and consider an action of $\mathbb{F}_r$ on $G$.  Then $\mathrm{C}^*_\lambda(G \rtimes \mathbb{F}_r)$ is MF and the canonical trace on $\mathrm{C}^*_\lambda(G \rtimes \mathbb{F}_r)$ is MF.
\end{theorem}

\begin{remark}
Using that the canonical trace on  $\mathrm{C}^*_\lambda(G \rtimes \mathbb{F}_r)$ is MF, one can easily show the group von Neumann algebra $L(G \rtimes \mathbb{F}_r)$ embeds into $\mathcal{R}^\omega$, the tracial ultrapower of the hyperfinite $\mathrm{II}_1$ factor $\mathcal{R}$, and hence $G \rtimes \mathbb{F}_r$ is hyperlinear (see Proposition \ref{prop:ConnesEmbedding} below).  After a preliminary version of this paper was announced, we were informed by Ben Hayes that the hyperlinearity of $G \rtimes \mathbb{F}_r$ follows from known results obtained in~\cite{BrownDykemaJung}. Also, the class of groups $G$ for which the canonical trace on C$^*_\lambda(G)$ is MF has also been considered in \cite{Hayes}. In particular, see Proposition 3.11 in \cite{Hayes} where a list of permanence properties are given for this class of groups.
\end{remark}

The paper is organized as follows.  Section \ref{sec:Preliminaries} establishes some notation and basic facts about ultrapowers and approximate morphisms which will be used throughout the paper.  Section \ref{sec:MFAlgebras} contains the definition of MF algebras and traces along with dynamical analogues of these notions.  Also, some new permanence properties of these algebras are obtained.  Section \ref{sec:KTheoreticDynamics} develops a systematic treatment of approximation properties for K-theoretic dynamics.  In some sense, we show there is no K-theoretic obstruction to a stably finite algebra being MF and identify the only K-theoretic obstruction to having an MF crossed product.  Section \ref{sec:ClassificationStuff} collects the consequences of the classification program needed to lift the K-theoretic approximations obtained in Section \ref{sec:KTheoreticDynamics} to norm approximations.  In Section \ref{sec:Proofs} the results stated in this introduction are proved.  The final section gives a version of Theorem \ref{thm:MainResult} for a certain class of Cuntz-Pimsner algebras.

\section{Notation and Preliminaries}
\label{sec:Preliminaries}

Throughout this piece, $G$ is a discrete group with neutral element $e$.  Also, we will restrict to countable groups and separable, unital C$^*$-algebras whenever it is convenient.  The set of projections in a \cstar-algebra $A$ is denoted by $\mathcal{P}(A)$, whereas $\mathcal{U}(A)$ refers to the group of a unitaries in $A$.  More generally, for $n \geq 1$, let $\mathcal{P}_n(A)$ and $\mathcal{U}_n(A)$ denote the collection of projections and unitaries in $\mathbb{M}_n(A)$, respectively, and define $\mathcal{P}_\infty(A) := \amalg_{n \geq 1} \, \mathcal{P}_n(A)$ and $\mathcal{U}_\infty(A) := \amalg_{n \geq 1} \, \mathcal{U}_n(A)$.  The free group on $r$ generators is written as $\mathbb{F}_r$, where $r\in\{1,2,\dots,\infty\}$, and the universal UHF-algebra is denoted by $\mathcal{Q}$.

The minimal tensor product of C$^*$-algebras $A$ and $B$ is denoted by $A \otimes B$.  By a \cstar-dynamical system (or an action $\alpha:G\curvearrowright A$), we mean a triple $(A, G, \alpha)$ where $G$ is a group, $A$ is a \cstar-algebra, and $\alpha:G \rightarrow\Aut(A)$ is a homomorphism into the group of automorphisms of $A$.  Given such an action, we write $A\rtimes_{\alpha}G$ for the full crossed product and $A \rtimes_{\alpha, \lambda} G$ for the reduced crossed product.  When the action $\alpha$ is understood, we may simply write $A \rtimes G$ and $A \rtimes_\lambda G$ for the full and reduced crossed products, respectively.  When $A$ is unital, the group $G$ is identified with its image in the crossed product via the unitary representation $G \hookrightarrow \mathcal{U}(A \rtimes_\lambda G)\ (s \mapsto 1_Au_s,\  s\in G)$. The algebra $A$ also embeds into the crossed product via the map $A \hookrightarrow A \rtimes_\lambda G\ (a\mapsto au_e)$.  Let $\mathbb{E} : A \rtimes_\lambda G \rightarrow A$ denote the canonical conditional expectation determined by $\mathbb{E}(u_s) = 0$ for $s \in G$ with $s \neq e$.

By a trace on a C$^*$-algebra, we mean a tracial state; i.e. a trace $\tau$ on $A$ is a state such that $\tau(ab) = \tau(ba)$ for every $a, b \in A$. The unique trace on the $\mathbb{M}_n$ is denoted by $\operatorname{tr}_n$, and $\operatorname{tr}_G$ refers to the canonical trace on C$_{\lambda}^{*}(G)$. Given a \cstar-dynamical system $(A,G,\alpha)$, a trace $\tau$ on $A$ is \emph{invariant} if $\tau \circ \alpha_s = \tau$ for all $s \in G$.  If $\tau$ is an invariant trace on $A$, then the state $\tau\circ\mathbb{E}:A\rtimes_{\lambda}G\rightarrow\mathbb{C}$ is also a trace.

Ultrapowers of \cstar-algebras and ordered abelian groups will surface regularly throughout this article. We outline their construction and prove some useful properties. To this end, let $\omega$ be a free ultrafilter on $\mathbb{N}$ which is fixed throughout the paper.

Given a C$^*$-algebra $B$, the C$^*$-algebra $\ell^\infty(B)$ consisting of bounded sequences in $B$ contains the subalgebra
\[c_\omega(B) := \left\{ (b_n)_n \in \ell^\infty(B) : \lim_{n \rightarrow \omega} \|b_n\| = 0 \right\}\]
as a closed, two-sided ideal. The \emph{norm ultrapower} of $B$ is the quotient \cstar-algebra
\[ B_\omega := \ell^\infty(B) / c_\omega(B). \]
We write $\pi_{\omega}: \ell^\infty(B) \rightarrow B_\omega$ for the quotient map.  If $B$ is equipped with a distinguished trace $\tau$, the norm ultrapower admits the trace $\tau_\omega$ defined by
\[\tau_{\omega}(\pi_{\omega}((b_n)_n))=\lim_{n\rightarrow\omega}\tau_{n}(b_n).\]

\begin{definition} Let $A$ and $B_n$ be separable, unital \cstar-algebras for $n \in \mathbb{N}$.  An \emph{approximate morphism} from $A$ to $(B_n)_{n\geq1}$ is a sequence of unital $\ast$-linear maps $(\varphi_n:A\rightarrow B_n)_{n\geq1}$ satisfying
\[\lim_{n\rightarrow\infty}\|\varphi_n(aa')-\varphi_n(a)\varphi_n(a')\|=0\qquad \text{for all $a, a' \in A$}. \]

An approximate morphism $(\varphi_n)_{n \geq 1}$ is \emph{faithful} if the maps further satisfy
\[\lim_{n\rightarrow\infty}\|\varphi_n(a)\|=\|a\|\qquad \text{for all $a\in A$}.\]

If $A$ and $B$ are equipped with distinguished traces $\tau$ and $\tau_n$, we call an approximate morphism $(\varphi_n)_{n\geq 1}$ \emph{trace-preserving} if
\[\lim_{n\rightarrow\infty}\tau_n(\varphi_n(a))=\tau(a)\qquad \text{for all $a\in A$}.\]
\end{definition}

The following standard fact will be used throughout the paper and will allow us to freely move between approximate morphisms and (exact) morphisms into ultrapowers.

\begin{proposition}\label{prop:approxmorphism}
Let $A$ and $B$ be separable, unital C*-algebras.
\begin{enumerate}
  \item An approximate morphism $(\varphi_n : A \rightarrow B)_{n \geq 1}$ induces a morphism $\varphi : A \rightarrow B_\omega$ such that $\varphi(a) = \pi_\omega((\varphi_n(a))_n)$ for all $a \in A$.  If $(\varphi_n)_{n\geq1}$ is faithful or trace-preserving, then $\varphi$ has the same property.
  \item Every morphism $\varphi : A \rightarrow B_\omega$ is induced by an approximate morphism $(\varphi_n)_{n \geq 1}$ as in part (1).  If $\varphi$ is faithful or trace-preserving, we may choose $(\varphi_n)_{n \geq 1}$ with the same property.
\end{enumerate}
\end{proposition}

\begin{proof}
In passing from the approximate morphism $(\varphi_n)_{n \geq 1}$ to the morphism $\varphi$, the only non-trivial claim is that the sequence $(\varphi_n(a))_n$ is an element of $\ell^\infty(B)$ for all $a \in A$.  This was shown by Connes and Higson in \cite{ConnesHigson:EThoery} in a slightly different context; we reproduce their argument here.  To this end, fix $a \in A$ and write $a^*a+x^*x=\|a\|^21_A$ for some $x\in A$. Then $\|a\|^21_B=\varphi_n(a^*a)+\varphi_n(x^*x)$, so
\begin{align*}
\|\varphi_n(a)^*\varphi_n(a)&+\varphi_n(x)^*\varphi_n(x)-\|a\|^21_B\|\\
&\leq\|\varphi_n(a)^*\varphi_n(a)-\varphi_n(a^*a)\|+\|\varphi_n(x)^*\varphi_n(x)-\varphi_n(x^*x)\|\rightarrow 0
\end{align*}
as $n \rightarrow \infty$.  For a tolerance $\varepsilon>0$ and sufficiently large $n$,
\[0\leq\varphi_n(a)^*\varphi_n(a)\leq \varphi_n(a)^*\varphi_n(a)+\varphi_n(x)^*\varphi_n(x)\leq (\|a\|^2+\varepsilon)1_B, \]
and hence $\|\varphi_n(a)\|^2\leq\|a\|^2+\varepsilon$ for all such $n$.  Therefore, $(\varphi_n(a))_{n \geq 1}$ is bounded.

To prove the second part of the proposition, fix a morphism $\varphi : A \rightarrow B_\omega$.  Let $(\Omega_k)_{k \geq 1}$ denote an increasing sequence of finite subsets of $A$ whose union is dense in $A$ and let $(\varepsilon_k)_{k \geq 1}$ be a decreasing sequence of positive numbers converging to 0.  Let $\psi : A \rightarrow \ell^\infty(B)$ be a unital, linear map with $\pi_\omega \circ \psi = \varphi$ and let $\psi_n : A \rightarrow B$ denote the components of $\psi$ so that $\psi(a) = (\psi_n(a))_n$ for all $a \in A$.  By replacing $\psi$ with the map
\[ A \rightarrow \ell^\infty(B) : a \mapsto \frac12 (\psi(a) + \psi(a^*)^*), \]
we may assume $\psi$ is $*$-preserving and hence each $\psi_n$ is $*$-preserving.

As $\varphi$ is a morphism, for any $k, m \in \mathbb{N}$, the set
\[ \{ n \in \mathbb{N} : n \geq m \text{ and } \| \psi_n(a a') - \psi_n(a) \psi_n(a') \| < \varepsilon_k \text{ for all $a, a' \in \Omega_k$} \} \]
is an element of $\omega$ and, in particular, is non-empty.  Hence one may inductively construct an increasing sequence $(n_k)_{k \geq 1}$ of natural numbers converging to $\infty$ such that
\[ \| \psi_{n_k}(aa') - \psi_{n_k}(a) \psi_{n_k}(a') \| < \varepsilon_k \qquad \text{ for all $a, a' \in \Omega_k$}. \]
Define $\varphi_n = \psi_{n_k}$ for each $n_k \leq n < n_{k+1}$ and $k \in \mathbb{N}$, and for $1 \leq n < n_1$, choose a unital, $*$-linear map $\varphi_n : A \rightarrow B$ arbitrarily.

Verifying that $(\varphi_n)_{n \geq 1}$ is an approximate morphism and induces the given morphism $\varphi : A \rightarrow B_\omega$ is routine.  In the case where $\varphi$ is faithful (trace-preserving), one can choose the $n_k \in \mathbb{N}$ as above such that $\psi_{n_k}$ are approximately isometric (trace-preserving) in addition to being approximately multiplicative.  Then the approximate morphism $(\varphi_n)_{n \geq 1}$ is faithful (trace-preserving).
\end{proof}

\begin{remark}
In the proposition above, the ultrapower $B_\omega$ can be replaced with the corona algebra $B_\infty := \ell^\infty(B) / c_0(B)$ with a similar proof. If one is only interested in faithful approximate morphisms, there is very little to be gained from considering the ultrapower $B_\omega$ in place of the corona algebra $B_\infty$.  There is, however, a technical advantage in using $B_\omega$ to encode trace-preserving approximate morphisms.  There is no canonical way to obtain a trace on $B_\infty$ from a trace $\tau$ on $B$.  On the other hand, a trace $\tau$ on $B$ induces a trace $\tau_\omega$ on $B_\omega$ by taking limits along $\omega$.
\end{remark}

Recall, an \emph{ordered abelian group} $(K, K^+)$ consists of a abelian group $K$ together with a semigroup $K^+ \subseteq K$ generating $K$ such that $K^+ \cap (-K^+) = \{0\}$.  As usual, for $x, y \in K$ we write $x \leq y$ to indicate $y - x \in K^+$.  An \emph{order unit} $u \in K^+$ is an element such that for all $x \in K$, there is an $n \in \mathbb{N}$ with $-n u \leq x \leq n u$.  We always assume an ordered abelian group to be equipped with a distinguished order unit $u$, and we often write $K$ in place of the triple $(K, K^+, u)$.  A \emph{morphism} $\sigma : K \rightarrow L$ of ordered abelian groups is an group morphism $\sigma : K \rightarrow L$ such that $\sigma(K^+) \subseteq L^+$ and $\sigma(u_K) = u_L$.  The morphism $\sigma$ is \emph{faithful} if $\ker(\sigma) \cap K^+ = \{0\}$.  A \emph{state} on $K$ is a morphism $\mu : K \rightarrow (\mathbb{R}, \mathbb{R}^+, 1)$.

If $A$ is a stably finite, unital \cstar-algebra, then K$_0(A)$ is an ordered abelian group with the positive cone determined by the projections in $\mathcal{P}_\infty(A)$ and the order unit determined by the unit $1_A$.  Note that a morphism $\varphi : A \rightarrow B$ between stably finite, unital C*-algebras induces a morphism $\mathrm{K}_0(\varphi) : \mathrm{K}_0(A) \rightarrow \mathrm{K}_0(B)$ of ordered abelian groups and $\mathrm{K}_0(\varphi)$ is faithful whenever $\varphi$ is faithful.  Similarly, a trace $\tau$ on $A$ induces a state $\mathrm{K}_0(\tau)$ on $\mathrm{K}_0(A)$ which is faithful if $\tau$ is faithful.  Moreover, is $A$ and $B$ are equipped with distinguished traces and $\varphi : A \rightarrow B$ is a trace-preserving morphism, then $\mathrm{K}_0(\varphi) : \mathrm{K}_0(A) \rightarrow \mathrm{K}_0(B)$ is a state-preserving morphism.

As with C$^*$-algebras, we will need the notion of ultrapowers and approximate morphisms of ordered abelian groups.  Given an ordered abelian group $L$, let $\ell^\infty(L)$ denote the collection of sequences $(x_n)_n \subseteq L$ which are bounded in the sense that there is a $d \in \mathbb{N}$ such that $- du \leq x_n \leq du$ for all $n \in \mathbb{N}$.  Then $\ell^\infty(L)$ is an ordered group when equipped with the component-wise ordering and the order unit given by the constant sequence $u$.  Define the subgroup $c_\omega(L) \subseteq \ell^\infty(K)$ by
\[ c_\omega(L) = \{ (x_n)_n \in \ell^\infty(L) : \text{ the set } \{n \in \mathbb{N} : x_n = 0 \} \text{ is in } \omega \} \]
and let $L_\omega$ denote the quotient $\ell^\infty(L) / c_\omega(L)$ and let $\kappa_\omega : \ell^\infty(L) \rightarrow L_\omega$ denote the quotient map. Then $L_\omega$ is an ordered abelian group with $L_\omega^+ = \kappa_\omega(\ell^\infty(L)^+)$ and order unit $u_{L_\omega} = \kappa_\omega(u_{\ell^\infty(L)})$.  Given a state $\mu$ on $L$, there is an induced state $\mu_\omega$ on $L_\omega$ given by
\[ \mu_\omega(\kappa_\omega((x_n)_n)) = \lim_{n \rightarrow \omega} \mu(x_n). \]

\begin{definition}
Let $K$ and $L_n$ be countable ordered abelian groups for $n \in \mathbb{N}$.  An \emph{approximate morphism} from $K$ to $(L_n)_{n\geq1}$ is a sequence $(\lambda_n: K \rightarrow L_n)_{n\geq1}$ of unit-preserving maps such that
\begin{enumerate}
  \item if $x, y \in K$, then $\lambda_n(x + y) = \lambda_n(x) + \lambda_n(y)$ for all sufficiently large $n$, and
  \item if $x \in K^+$, then $\lambda_n(x) \geq 0$ for all sufficiently large $n$.
\end{enumerate}

The approximate morphism $(\lambda_n)_{n\geq1}$ is called \emph{faithful} if the following holds: if $x \in K^+ \setminus \{0\}$, then $\lambda_n(x) \neq 0$ for all sufficiently large $n$.

When $K$ and $L_n$ are equipped with distinguished states $\mu$ and $\mu_n$, the approximate morphism $(\lambda_n)_{n\geq1}$ is called \emph{state-preserving} if, for all $x \in K$, $\displaystyle \lim_{n \rightarrow \infty} \mu_n(\lambda_n(x)) = \mu(x)$.
\end{definition}

The following is an exact analogue of Proposition \ref{prop:approxmorphism}.

\begin{proposition}\label{prop:groupapproxmorphism}
Let $K$ and $L$ be countable ordered abelian groups.
\begin{enumerate}
  \item An approximate morphism $(\lambda_n : K \rightarrow L)_{n \geq 1}$ induces a morphism $\lambda : K \rightarrow L_\omega$ such that $\lambda(x) = \kappa_\omega((\lambda_n(x))_n)$ for all $x \in K$.  If $(\lambda_n)$ is faithful or trace-preserving, then $\lambda$ has the same property.
  \item Every morphism $\lambda : K \rightarrow L_\omega$ is induced by an approximate morphism $(\lambda_n)_{n \geq 1}$ as in part (1).  If $\lambda$ is faithful or trace-preserving, we may choose $(\lambda_n)_{n \geq 1}$ with the same property.
\end{enumerate}
\end{proposition}

\begin{proof}
The proof is nearly identical to the proof of Proposition \ref{prop:approxmorphism}, so we omit most of the details.  The only somewhat subtle issue is showing that if $(\lambda_n)_{n\geq1}$ is an approximate morphism from $K$ to $L$, then for every $x \in K$, the sequence $(\lambda_n(x))_n$ is in $\ell^\infty(L)$.  To see this, note that there is a $d \in \mathbb{N}$ with $- d u_K \leq x \leq d u_K$.  Since $d u_K - x \geq 0$, for all sufficiently large $n$,
\[ 0 \leq \lambda_n(d u_K - x) = d \lambda_n(u_K) - \lambda_n(x) = d u_L - \lambda_n(x), \]
and hence $\lambda_n(x) \leq d u_L$ for all large $n$.  Similarly, $-du_L \leq \lambda_n(x)$ for all large $n$.  Hence, by increasing $d$ if necessary, we have $-d u_L \leq \lambda_n(x) \leq d u_L$ for all $n \in \mathbb{N}$ and $(\lambda_n(x))_{n\geq 1}$ is bounded as claimed.
\end{proof}

We end this preliminary section with a result that identifies the the K$_0$-group of a norm ultrapower of the UHF alegbra $\mathcal{Q}$ with the ordered group ultrapower of $\mathrm{K}_0(\mathcal{Q})$.  This is a special case of a result due to Dadarlat and Eilers (see Section 3.2 in \cite{DadarlatEilers}), however, we give a proof in our context.  Recall that there is an isomorphism
\[ (\mathrm{K}_0(\mathcal{Q}),\mathrm{K}_0(\mathcal{Q})^+,[1_{\mathcal{Q}}]_0) \overset{\cong}{\longrightarrow} (\mathbb{Q},\mathbb{Q}^+, 1) \]
induced by the unique trace on $\mathcal{Q}$.

\begin{proposition}\label{prop:K0UHFUltrapower}
There is an isomorphism $\theta: \mathrm{K}_0(\mathcal{Q}_\omega) \rightarrow \mathrm{K}_0(\mathcal{Q})_{\omega}\cong\mathbb{Q}_{\,\omega}$ of ordered abelian groups satisfying
\[\theta\left([\pi_\omega((p_n)_n)]_0\right)=\kappa_\omega\left(([p_n]_0)_n\right),\quad p_n\in\mathcal{P}(\mathcal{Q}).\]
\end{proposition}

\begin{proof}
To ease notation we will identify $\mathbb{M}_{d}(\ell^\infty(\mathcal{Q}))$ with $\ell^\infty(\mathbb{M}_{d}(\mathcal{Q}))$ for each $d\in\mathbb{N}$. It is routine to verify there is a morphism $\tilde{\theta} : \mathrm{K}_0(\ell^\infty(\mathcal{Q})) \rightarrow \ell^\infty(\mathcal{K}_0(\mathcal{Q}))$ such that $\tilde{\theta}([(p_n)_n]_0) = ([p_n]_0)_n$ for all projections $(p_n)_n \in \mathbb{M}_d(\ell^\infty(\mathcal{Q}))$ and $d \in \mathbb{N}$.  We claim $\tilde{\theta}$ is an isomorphism of ordered abelian groups.

To show $\tilde{\theta}$ is injective, suppose $p = (p_n)_n$ and $q = (q_n)_n$ are projections in $\mathbb{M}_d(\ell^\infty(\mathcal{Q}))$ with $\tilde{\theta}([p]_0) = \tilde{\theta}([q]_0)$.  Then $[p_n]_0 = [q_n]_0$ in $\mathrm{K}_0(\mathcal{Q})$ for all $n \in \mathbb{N}$.  Since $\mathcal{Q}$ has cancellation, there is, for all $n \in \mathbb{N}$, a partial isometry $v_n \in \mathbb{M}_d(\mathcal{Q})$ with initial projection $p_n$ and final projection $q_n$.  Now, $v = (v_n)_n$ is a partial isometry in $\mathbb{M}_d(\ell^\infty(\mathcal{Q}))$ with initial projection $p$ and final projection $q$.  Hence $[p]_0 = [q]_0$ and $\tilde{\theta}$ is injective.

On the other hand, given $x = (x_n)_n \in \ell^\infty(\mathrm{K}_0(\mathcal{Q}))^+$, there is a $d \in \mathbb{N}$ such that $0 \leq x_n \leq d [1]_0$ for all $n \in \mathbb{N}$.  Again, as $\mathcal{Q}$ has cancellation, there are projections $p_n \in \mathbb{M}_d(\mathcal{Q})$ with $[p_n]_0 = x_n$ in $\mathrm{K}_0(\mathcal{Q})$ for all $n \in \mathbb{N}$.  Then $p = (p_n)_n$ is a projection in $\mathbb{M}_d(\ell^\infty(\mathcal{Q}))$ and $\tilde{\theta}([p]_0) = x$.  Hence $\tilde{\theta}$ is surjective.  Note that the proof also shows the inverse of $\tilde{\theta}$ is positive and hence $\tilde{\theta}$ is an isomorphism of ordered abelian groups.

Considerations similar to those above shows $\tilde{\theta}$ restricts to a surjective map $\tilde{\theta}_0$ from $\mathrm{K}_0(c_\omega(\mathcal{Q}))$ onto $c_\omega(\mathrm{K}_0(\mathcal{Q}))$.  As projections in $\mathbb{M}_d(\mathcal{Q}_\omega)$ lift to projections in $\mathbb{M}_d(\ell^\infty(\mathcal{Q}))$, we have a commutative diagram
\[ \begin{tikzcd}
        & \mathrm{K}_0(c_\omega(\mathcal{Q})) \arrow{r} \arrow{d}{\tilde{\theta}_0} & \mathrm{K}_0(\ell^\infty(\mathcal{Q})) \arrow{r}{\mathrm{K}_0(\pi_\omega)} \arrow{d}{\tilde{\theta}} & \mathrm{K}_0(\mathcal{Q}_\omega) \arrow{r} \arrow[dotted]{d}{\exists \, \theta} & 0 \\ 0 \arrow{r} & c_\omega(\mathrm{K}_0(\mathcal{Q})) \arrow{r} & \ell^\infty(\mathrm{K}_0(\mathcal{Q})) \arrow{r}{\kappa_\omega} & \mathrm{K}_0(\mathcal{Q})_\omega \arrow{r} & 0
\end{tikzcd} \]
with exact rows.  As $\tilde{\theta}$ is an isomorphism and $\tilde{\theta}_0$ is surjective, a diagram chase shows the existence of the desired isomorphism $\theta : \mathrm{K}_0(\mathcal{Q}_\omega) \rightarrow \mathrm{K}_0(\mathcal{Q})_\omega$.
\end{proof}

\section{MF Algebras, MF Traces, and Permanence Properties}\label{sec:MFAlgebras}

Quasidiagonality was originally introduced by Halmos in the context of single operator theory.  The notion was imported to C$^*$-algebra theory by Voiculescu and was defined in terms of representations of C$^*$-algebras as quasidiagonal operators.  Voiculescu also obtained a characterization of quasidiagonal C$^*$-algebras in terms of completely positive finite rank approximations (see Chapter 8 of \cite{BrownOzawa} for a good introduction).  By removing the complete positivity assumptions from the approximations, one obtains the MF algebras of Blackadar and Kirchberg introduced in \cite{BlackadarKirchberg} (see also Chapter 11 of \cite{BrownOzawa}).

Building on Alain Connes's work on injective factors, N. Brown introduced in \cite{BrownQDTraces} the notion of a quasidiagonal trace by replacing the faithful approximate morphisms with trace-preserving approximate morphisms.  Since their conception, quasidiagonal traces have played a fundamental role in the classification and regularity theories of nuclear C$^*$-algebras.  Perhaps the most definitive result in this direction was obtained in \cite{ElliottGongLinNiu} where it was shown that if $A$ is a separable, simple, unital C$^*$-algebra which satisfies the UCT, has finite nuclear dimension, and every trace on $A$ is quasidiagonal, then $A \otimes \mathcal{Q}$ admits certain internal approximations, which is in fact enough to show that $A$ is classifiable in the sense of Elliott as was shown in \cite{GongLinNiu}.

By removing complete positivity from the approximations, we introduce the notion of an MF trace on a C$^*$-algebra.  Whereas quasidiagonal traces appear to be closely related to amenability, there is no clear obstruction for a trace to be MF.  We now state formally the definitions of QD and MF for both C$^*$-algebras and traces on C$^*$-algebras.

\begin{definition}\label{DefnMF} Let $A$ be a C$^*$-algebra.
\begin{enumerate}
  \item A C$^*$-algebra $A$ is called \emph{matricial field (MF)} if for every finite set $\Omega \subseteq A$ and $\varepsilon > 0$, there is an $n \geq 1$ and a $\ast$-linear map $\varphi : A \rightarrow \mathbb{M}_n$ such that
  \begin{enumerate}
    \item $\|\varphi(ab) - \varphi(a)\varphi(b) \| < \varepsilon$, and
    \item $\big| \| \varphi(a) \| - \|a\| \big| < \varepsilon$.
  \end{enumerate}
  for every $a, b \in \Omega$.
\item A trace $\tau$ on $A$ is called \emph{matricial field (MF)} if for every finite set $\Omega \subseteq A$ and $\varepsilon > 0$, there is an $n \geq 1$ and a $\ast$-linear map $\varphi : A \rightarrow \mathbb{M}_n$ satisfying
  \begin{enumerate}
    \item $\|\varphi(ab) - \varphi(a)\varphi(b) \| < \varepsilon$, and
    \item $|\tau(a) - \operatorname{tr}_n(\varphi(a)) | < \varepsilon$,
  \end{enumerate}
  for every $a, b \in \Omega$.
\end{enumerate}
An algebra $A$ or a trace $\tau$ is \emph{quasidiagonal (QD)} if the map $\varphi$ can be chosen to be completely positive and contractive.
\end{definition}

A few observations are in order. When $A$ is unital the map $\varphi$ can be chosen to be unital. In the separable case, $A$ is MF if and only if there is a faithful approximate morphism $(\varphi_n:A\rightarrow\mathbb{M}_{k_n})_{n\geq1}$. Similarly, a trace $\tau$ on $A$ is MF if and only if there is a trace-preserving approximate morphism $(\varphi_n:A\rightarrow\mathbb{M}_{k_n})_{n\geq1}$.

Some permanence properties of MF algebras and traces are fairly straightforward. For instance, it is clear that the MF and QD properties pass to subalgebras. Similarly, if $\tau$ is an MF (resp. QD) trace on a unital C$^*$-algebra $A$, and $B \subseteq A$ is a unital subalgebra, then $\tau|_B$ is an MF (resp. QD) trace.  Also, $A$ is MF (resp. QD) if and only if for every separable subalgebra $B \subseteq A$, the algebra $B$ is MF (resp. QD); the analogous result for traces also holds.

The next two propositions give some other easy permanence properties of MF algebras and MF traces.  In this section, we consider more delicate permanence properties for MF algebras and traces such as tensor products, crossed products, quotients, extensions, and a Mayer-Vietoris-type property.

\begin{proposition}\label{prop:permMF}
Let $A$ be a separable C*-algebra.
\begin{enumerate}
  \item A is MF if and only if there is an embedding $A \hookrightarrow \mathcal{Q}_\omega$.
  \item If $(B_n)_{n \geq 1}$ is a sequence of MF algebras and there is a faithful approximate morphism $(A\rightarrow B_n)_{n \geq 1}$, then $A$ is also MF.
  \item If $A$ is an MF algebra, then $A_\omega$ is MF.
\end{enumerate}
\end{proposition}

\begin{proof}
To show (2), let $\Omega\subset A$ be a finite set and $\varepsilon>0$. Let $(\psi_n)_{n \geq 1}$ denote the given approximate morphism and choose $n \in \mathbb{N}$ such that $\psi_n:A\rightarrow B_n$ is approximately multiplicative and isometric within $\varepsilon/3$ on $\Omega$.  Also, choose a $\ast$-linear map $\rho:B_n \rightarrow \mathbb{M}_{d}$ that is approximately multiplicative and isometric within $\varepsilon/3$ on
\[ \Omega':=\psi_n(\Omega)\cup\left\{\psi_n(aa')-\psi_n(a)\psi_n(a')\ |\ a,a'\in\Omega\right\}.\]
The map $\varphi:=\rho\circ\psi_n:A\rightarrow\mathbb{M}_{d}$ is $\ast$-linear, and, for $a,a'\in\Omega$,
\begin{align*}
  \|\varphi(aa')-&\varphi(a)\varphi(a')\| \\
  &\leq \|\rho\left(\psi_n(aa')-\psi_n(a)\psi_n(a')\right)\|+\|\rho(\psi_n(a)\psi_n(a'))-\rho(\psi_n(a))\rho(\psi_n(a'))\|\\
  &\leq \|\psi_n(aa')-\psi_n(a)\psi_n(a')\|+\varepsilon/3+\varepsilon/3<\varepsilon.
\end{align*}
Also,
\[\big|\|\varphi(a)\|-\|a\|\big|\leq\big|\|\rho\left(\psi_n(a)\right)\|-\|\psi_n(a)\|\big|+\big|\|\psi_n(a)\|-\|a\|\big|<2\varepsilon/3<\varepsilon.\]
for every $a\in\Omega$. This shows that $A$ is MF.

For (3), note that a C$^*$-algebra is MF if and only if every separable C$^*$-subalgebra is MF.  Given a separable subalgebra $B \subseteq A_\omega$, the inclusion $B \hookrightarrow A_\omega$ lifts to a faithful approximate morphism $(B \rightarrow A)_{n\geq1}$ by Proposition \ref{prop:approxmorphism}.  Hence $B$ is MF by (3) and therefore $A_\omega$ is MF.

Finally, for (1), if $A$ is MF, there is a faithful approximate morphism $(\varphi_n : A \rightarrow \mathbb{M}_{k_n})_{n \geq 1}$ for integers $k_n \in \mathbb{N}$.  Composing the $\varphi_n$ with embeddings $M_{k_n} \subseteq \mathcal{Q}$ produces a faithful approximate morphism from $A$ to $\mathcal{Q}$. By Proposition \ref{prop:approxmorphism}, there is a faithful morphism $A \rightarrow \mathcal{Q}_\omega$.

Conversely, fix an increasing sequence $(B_n)_{n \geq 1}$ of finite dimensional subalgebras of $\mathcal{Q}$ with dense union.  By Arveson's Extension Theorem, there are conditional expectations $\mathbb{E}_n : \mathcal{Q} \rightarrow B_n$.  The sequence $(\mathbb{E}_n)_{n\geq1}$ is a faithful approximate morphism from $\mathcal{Q}$ into finite dimensional C$^*$-algebras and hence $\mathcal{Q}$ is MF.  By part (3), $\mathcal{Q}_\omega$ is MF and the result follows since subalgebras of MF algebras are MF.
\end{proof}

The following is an exact analogue of Proposition \ref{prop:permMF} for MF traces.  The proof is nearly identical and is left to the reader.

\begin{proposition}\label{prop:permMFTrace}
Let $A$ be a separable C*-algebra and let $\tau$ be a trace on $A$.
\begin{enumerate}
  \item $\tau$ is MF if and only if there is a trace-preserving morphism $A \rightarrow \mathcal{Q}_\omega$.
  \item If $(B_n)_{n\geq1}$ is a sequence of C*-algebras with MF traces $\tau_n$ on $B_n$ and there is a trace-preserving approximate morphism $(A \rightarrow B_n)_{n\geq 1}$, then $\tau$ is MF.
  \item If $\tau$ is MF, then the trace $\tau_\omega$ on $A_\omega$ is MF.
\end{enumerate}
\end{proposition}

As eluded to in the introduction, the question of which algebras are MF can be thought of as the C$^*$-analogue of Connes's Embedding Problem.  The connection is made explicit in the result below.

\begin{proposition}\label{prop:ConnesEmbedding}
Let $A$ be a separable C*-algebra with a tracial state $\tau$ and let $\pi_\tau$ denote the GNS representation of $\tau$.  If $\tau$ is MF, then $\pi_\tau(A)''$ admits a trace-preserving embedding into $\mathcal{R}^\omega$, the tracial ultrapower of the hyperfinite $\mathrm{II}_1$ factor $\mathcal{R}$.
\end{proposition}

\begin{proof}
Let $\tau_\omega$ denote the trace on $\mathcal{Q}_\omega$ induced by the unique trace on $\mathcal{Q}$ and let $J = \{ a \in \mathcal{Q}_\omega : \tau_\omega(a^*a) = 0 \}$.  There is a natural embedding $\mathcal{Q} \subseteq \mathcal{R}$ induced by the GNS representation of the unique trace on $\mathcal{Q}$, and the unit ball of $\mathcal{Q}$ is 2-norm dense in the unit ball of $\mathcal{R}$.  It follows that $\mathcal{Q}_\omega / J \cong \mathcal{R}^\omega$ by Kaplansky's Density Theorem.  Hence and trace-preserving morphism $\varphi : A \rightarrow \mathcal{Q}_\omega$ induces a trace-preserving morphism $\pi_\tau(A)'' \rightarrow \mathcal{R}^\omega$ which is faithful by the faithfulness of the trace on $\pi_\tau(A)''$.
\end{proof}

\subsection{Tensor Products.}
It is well known that the minimal tensor product of two QD C$^*$-algebras is QD (Proposition 7.1.12 in \cite{BrownOzawa}).  The corresponding statement for QD traces is also true (Proposition 3.5.7 in \cite{BrownQDTraces}).  In the MF setting, these questions are more subtle.  The key difference is that the algebraic tensor product of two completely positive maps will always be continuous with respect to the minimal tensor norm, whereas the corresponding statement fails for tensor products of $\ast$-linear maps.  Blackadar and Kirchberg showed that the minimal tensor product of two MF algebras is again MF provided that at least one algebra is nuclear (Proposition 3.3.6 in \cite{BlackadarKirchberg}). Here we weaken the nuclearity condition. First we single out a standard tensor product result.

\begin{proposition}\label{prop:TensorsAndUltrapowers}
Suppose $A$ and $B$ are C*-algebras and $B$ is exact.  There is an embedding
\[ A_\omega \otimes B \hookrightarrow (A \otimes B)_\omega \]
given on generators by $\pi_{\omega}\left((a_n)_n\right) \otimes b \mapsto \pi_{\omega}\left((a_n \otimes b)_n\right)$.
\end{proposition}

\begin{proof}
Choose faithful representations $A \subseteq \mathbb{B}(\mathcal{H})$ and $B \subseteq \mathbb{B}(\mathcal{K})$.  Then there are induced embeddings
\[ \ell^\infty(A) \otimes B \hookrightarrow \mathbb{B}(\ell^2(H) \otimes \mathcal{K}) \quad \text{and} \quad
   \ell^\infty(A \otimes B) \hookrightarrow \mathbb{B}( \ell^2(\mathcal{H} \otimes \mathcal{K})). \]
Conjugating by the unitary
\[ \ell^2(H) \otimes K \rightarrow \ell^2(H \otimes K) \quad (\xi_n)_n \otimes \eta \mapsto (\xi_n \otimes \eta)_n \]
yields an embedding
\[ j: \ell^\infty(A) \otimes B \hookrightarrow \ell^\infty(A \otimes B). \]
Since $B$ is exact, we have a commutative diagram
\[ \begin{tikzcd}
   0 \arrow{r} & c_\omega(A) \otimes B \arrow{r} \arrow{d} & \ell^\infty(A) \otimes B \arrow{r} \arrow{d}{j} & A_\omega \otimes B \arrow{r} \arrow{d}{\bar{j}} & 0 \\
   0 \arrow{r} & c_\omega(A\otimes B) \arrow{r} & \ell^\infty(A \otimes B) \arrow{r} & (A \otimes B)_\omega \arrow{r} & 0
\end{tikzcd} \]
with exact rows.

The proof will be complete if we can show $\bar{j}$ is injective.  By Kirchberg's Slice Lemma \cite[Lemma 4.1.9]{RordamYellowBook}, if $\ker(\bar{j})$ is non-zero, then $\ker(\bar{j})$ contains a non-zero elementary tensor.  Given $a \in A_\omega$ and $b \in B$, let $(a_n)_n \in \ell^\infty(A)$ be a lift of $a$ and note that
\[ \| \bar{j}(a \otimes b) \| = \lim_{n \rightarrow \omega} \| a_n \otimes b \| = \| b \| \lim_{n \rightarrow \omega} \| a_n \| = \|a \| \| b\|. \]
Thus $\ker(\bar{j})$ contains no non-zero elementary tensors and hence $\bar{j}$ is injective.
\end{proof}

\begin{proposition}\label{prop:TensorProduct}
Let $A$ and $B$ be C*-algebras and suppose $B$ is exact.  If $A$ and $B$ are MF, then $A \otimes B$ is again MF.  Similarly, if $\tau_A$ and $\tau_B$ are MF traces on $A$ and $B$, then the trace $\tau_A \otimes \tau_B$ on $A \otimes B$ is also MF.
\end{proposition}

\begin{proof}
We only prove the norm version as the tracial version is nearly identical.  Since $A$ is MF, there is an embedding $A \hookrightarrow \mathcal{Q}_\omega$.  Since $B$ is exact, Proposition \ref{prop:TensorsAndUltrapowers} yields an embedding
\[ A \otimes B \hookrightarrow \mathcal{Q}_\omega \otimes B \hookrightarrow (Q \otimes B)_\omega. \]
As $\mathcal{Q}$ is nuclear, $\mathcal{Q} \otimes B$ is a MF by Proposition 3.3.6 of \cite{BlackadarKirchberg}. Proposition \ref{prop:permMF} now implies that $(Q \otimes B)_\omega$ is MF and the result follows since the MF property passes to subalgebras.
\end{proof}

\begin{remark}\label{rem:MFTensorProduct}
It is not clear if the exactness of $B$ is necessary in Proposition \ref{prop:TensorProduct}.  In general, it is not known if the minimal tensor product of two MF algebras is stably finite.
\end{remark}

\subsection{Crossed Products.}
In \cite{KerrNowak}, Kerr and Nowak introduced the notion of a quasidiagonal action of a group $G$ on a C$^*$-algebra $A$ and showed that if C$^*_\lambda(G)$ is MF and the action is quasidiagonal, then the reduced crossed product $A\rtimes_{\lambda}G$ is MF (Theorem 3.4 in \cite{KerrNowak}).  In \cite{Rainone}, the same result was proven using a weaker definition of a quasidiagonal action and a converse was obtained.  The notion of an MF action was introduced by the first author in \cite{Rainone}.  In full generality, it is not known if an MF action of a group $G$ whose reduced \cstar-algebra is MF yields an MF reduced  crossed product. Theorem \ref{thm:MFCrossedProducts} below provides a positive solution to this question in the case when the acting group is exact and also exhibits a tracial version of this result.

\begin{definition}\label{defn:MFaction}
Let $\alpha:G\rightarrow\Aut(A)$ be an action of a discrete group $G$ on a C$^*$-algebra $A$.  We say $\alpha$ is an \emph{MF-action} if, for all finite subsets $\Omega \subseteq A$ and $F \subseteq G$ and for all $\varepsilon >0$, there is an $n\in\mathbb{N}$, a $\ast$-linear map $\varphi : A \rightarrow \mathbb{M}_n$, and a function $w : G \rightarrow \mathcal{U}(\mathbb{M}_n)$ such that
\begin{enumerate}
  \item $\| w_{st} - w_s w_t \| < \varepsilon$,
  \item $\| w_s \varphi(a) w_s^* - \varphi(\alpha_s(a)) \| < \varepsilon$,
  \item $\| \varphi(ab) - \varphi(a) \varphi(b) \| < \varepsilon$, and
  \item $\left| \| \varphi(a) \| - \| a \| \right|< \varepsilon $
\end{enumerate}
for all $a, b \in \Omega$ and $s, t \in F$.

An invariant trace $\tau$ on $A$ is called \emph{$\alpha$-MF} if given $\Omega$, $F$, and $\varepsilon$ as above, there are $n$, $\varphi$, and $w$ as above satisfying (1), (2), (3), and
\begin{enumerate}
  \item[(4)'] $\left| \operatorname{tr}_n(\varphi(a)) - \tau(a)\right| < \varepsilon$ for all $a \in \Omega$.
\end{enumerate}
Quasidiagonal actions and $\alpha$-quasidiagonal traces are defined by further requiring $\varphi$ to be completely positive and contractive.
\end{definition}

As with MF algebras, note that when $A$ is separable and $G$ is countable, the action $\alpha$ is MF if and only if there is a faithful, covariant representation of $(A, G, \alpha)$ on $\mathcal{Q}_\omega$.  That is, there is an embedding $\varphi : A \rightarrow \mathcal{Q}_\omega$ and a unitary representation $w : G \rightarrow \mathcal{U}(\mathcal{Q}_\omega)$ with $\varphi \circ \alpha_s = \Ad_{w_s} \circ  \varphi$ for every $s \in G$. Similarly, an invariant trace $\tau$ on $A$ is MF if and only if there is a trace-preserving, covariant representation $\varphi: A \rightarrow \mathcal{Q}_\omega$.

\begin{theorem}\label{thm:MFCrossedProducts}
Suppose $(A,G,\alpha)$ is a C*-dynamical system with $G$ exact.
\begin{enumerate}
  \item If $\alpha$ is MF and $\mathrm{C}^*_\lambda(G)$ is MF, then $A \rtimes_{\lambda} G$ is MF.
  \item If $\tau$ is an $\alpha$-MF trace on $A$ and $\operatorname{tr}_G$ is MF, then the trace $\tau \circ \mathbb{E}$ on $A \rtimes_\lambda G$ is MF.
\end{enumerate}
\end{theorem}

\begin{proof}
We may assume $A$ is separable and $G$ is countable. Suppose $\varphi : A \rightarrow \mathcal{Q}_\omega$ is a $\ast$-homomorphism and $w : G \rightarrow \mathcal{U}(\mathcal{Q}_\omega)$ is a unitary representation with $\varphi \circ \alpha_s = \Ad_{w_s} \circ \varphi$ for all $s \in G$.  Let $\beta:G\curvearrowright\mathcal{Q}_{\omega}$ denote the action given via conjugation; that is $\beta_s:=\Ad_{w_s}$ for $s\in G$. Since $\beta$ is inner we know that $\mathcal{Q}_{\omega}\rtimes_{\lambda,\beta}G\cong\mathcal{Q}_{\omega}\otimes \mathrm{C}^*_\lambda(G)$. We therefore have a $\ast$-homomorphism
\[ A \rtimes_{\lambda, \alpha} G \overset{\tilde{\varphi}}{\longrightarrow} \mathcal{Q}_\omega \rtimes_{\lambda,\beta} G \cong \mathcal{Q}_\omega \otimes \mathrm{C}^*_\lambda(G).\]
If $\alpha$ is MF we can choose such a $w$ and $\varphi$ as above such that $\varphi$ is faithful, which then implies $\tilde{\varphi}$ is an embedding. Since $G$ is exact, $\mathrm{C}^*_\lambda(G)$ is exact, and (1) follows immediately from Propositions~\ref{prop:TensorProduct} and~\ref{prop:permMF}. Similarly, if $\tau$ is an $\alpha$-MF trace on $A$, we can choose $\varphi$ to be trace-preserving and, consequently, $\tilde{\varphi}:A\rtimes_{\lambda}G\rightarrow\mathcal{Q}_{\omega}\otimes\mathrm{C}^*_\lambda(G)$ is trace-preserving as well. Now (2) follows from the fact that $\tau_{\mathcal{Q}_{\omega}}\otimes\operatorname{tr}_G$ is MF by Propositions~\ref{prop:permMFTrace} and~\ref{prop:TensorProduct}.
\end{proof}

\begin{remark}
The converse of both results in Theorem \ref{thm:MFCrossedProducts} are true without assuming C$^*_\lambda(G)$ is exact.  The converse of (1) is proved in \cite[Proposition 3.4]{Rainone} and the same proof yields the converse of (2). We again emphasize that the condition of $G$ being exact in Theorem \ref{thm:MFCrossedProducts} may be superfluous (compare with Remark \ref{rem:MFTensorProduct}).  As for when a reduced crossed product is quasidiagonal, Proposition 3.4 in~\cite{Rainone} shows that if $A\rtimes_\lambda G$ is quasidiagonal, then C$^*_\lambda(G)$ and the action are quasidiagonal. It is also known that these conditions are sufficient when $A$ is nuclear (Theorem 3.19 in \cite{Rainone}). It is likely that nuclearity of $A$ is not necessary.
\end{remark}

It is not clear when C$^*_\lambda(G)$ and $\operatorname{tr}_G$ are MF.  A recent breakthrough by Tikuisis, White, and Winter shows this is the case whenever $G$ is amenable (Corollary C in \cite{TikuisisWhiteWinter}).  In the non-amenable case, Haagerup and Thorbj{\o}rnsen have shown that C$^*_\lambda(\mathbb{F}_r)$ is MF in \cite{FreeGroupMF}.  In particular, C$^*_\lambda(\mathbb{F}_r)$ has an MF trace.  A result of Powers shows $\operatorname{tr}_{\mathbb{F}_r}$ is the unique trace on C$^*_\lambda(\mathbb{F}_r)$ when $r \geq 2$ (Corollary VII.7.6 in \cite{DavidsonBook}) and hence $\operatorname{tr}_{\mathbb{F}_r}$ is MF.  It is also known that C$^*_\lambda(\mathbb{F}_r)$ is exact (Proposition 5.1.8 in \cite{BrownOzawa}) and hence Theorem \ref{thm:MFCrossedProducts} applies when $G$ is free.  We record this corollary below for future use.

\begin{corollary}\label{cor:FreeGroupCrossedProducts}
Let $\alpha : \mathbb{F}_r\rightarrow\Aut(A)$ be a given action.
\begin{enumerate}
  \item $A \rtimes_{\lambda} \mathbb{F}_r$ is MF if and only if $\alpha$ is MF.
  \item Given an  invariant trace $\tau$ on $A$, the trace $\tau \circ \mathbb{E}$ on $A \rtimes_\lambda \mathbb{F}_r$ is MF if and only if $\tau$ is $\alpha$-MF.
\end{enumerate}
\end{corollary}

The following permanence results utilize Proposition~\ref{prop:TensorProduct} and Theorem \ref{thm:MFCrossedProducts} as well as a well-known imprimitivity theorem of P.\ Green. Recall that a discrete group $G$ is said to be locally finite if for every finite set $F\subset G$, the subgroup $\langle F \rangle$ generated by $F$ is also finite. Locally finite groups are easily seen to be amenable.

\begin{proposition}\label{prop:MFcrosslocallyfinite}
Let  $\alpha:G\rightarrow\Aut(A)$ be an action of a locally compact group $G$ on an MF algebra $A$.
\begin{enumerate}
  \item If $G$ is a second countable compact group, then $A\rtimes_{\lambda}G= A\rtimes G$ is MF.
  \item If $G$ is locally finite then $A\rtimes_{\lambda}G=A\rtimes G$ is MF.
\end{enumerate}
\end{proposition}

\begin{proof}
The fact that $G$ is amenable in both (1) and (2) ensures that $A\rtimes_{\lambda}G=A\rtimes G$.


(1): Write $\beta$ for the action of $G$ on $C(G)$ induced by left translation; that is, $\beta_{s}(f)(t)=f(s^{-1}t)$ for $s,t\in G$ and $f\in C(G)$. Now consider the action $\alpha\otimes\beta:G\rightarrow \Aut(A\otimes C(G))$ given by $(\alpha\otimes\beta)_{s}=\alpha_s\otimes\beta_s$ for $s\in G$. The embedding $A \hookrightarrow A \otimes C(G)$ given by $a \mapsto a \otimes 1_{C(G)}$ induces an embedding.
\[A\rtimes_{\alpha}G\hookrightarrow (A\otimes C(G))\rtimes_{\alpha\otimes\beta}G.\]
From Corollary 2.8 of~\cite{Green} we get that $(A\otimes C(G))\rtimes_{\alpha\otimes\beta}G\cong A\otimes\mathbb{K}(L^2(G))$, where $\mathbb{K}(L^2(G))$ denotes the compact operators on $L^2(G)$. Since $A$ and $\mathbb{K}(L^2(G))$ are MF algebras and $\mathbb{K}(L^2(G))$ is nuclear  we know that $A\otimes\mathbb{K}(L^2(G))$ is also MF by Proposition~\ref{prop:TensorProduct}. The MF property passes to subalgebras so $A\rtimes_{\alpha}G$ is MF.

(2): Now consider a locally finite group $G$. Let $F\subset G$ and $\Omega\subset A$ be finite sets and let $\varepsilon>0$. Denote by $H$ the (finite) group generated by $F$ and by $\alpha^{H}$ the restricted action of $H$ on $A$. By part (1), that $A\rtimes_{\alpha^{H}}H$ is MF. The action $\alpha^{H}$, therefore, is MF, so there are $d \in \mathbb{N}$, a $\ast$-linear map $\phi: A\rightarrow\mathbb{M}_{d}$ and a map $w:H\rightarrow\mathcal{U}(\mathbb{M}_{d})$ with
\begin{enumerate}
\item $\|\phi(xy)-\phi(x)\phi(y)\|<\varepsilon$,
\item $\big|\|\phi(x)\|-\|x\|\big|<\varepsilon$,
\item $\|\phi(\alpha^{H}_{s}(x))-w_s\phi(x)w_s^*\|<\varepsilon$, and
\item $\|w_{st}-w_sw_t\|<\varepsilon$
\end{enumerate}
for every $x,y\in\Omega$ and $s,t\in F$. By extending the map $w$ arbitrarily to all of $G$, we conclude that the action $\alpha:G\curvearrowright A$ is MF. Now, $G$ is exact (actually amenable) and C$_{\lambda}^*(G)$ is MF (actually AF), and so, by Theorem~\ref{thm:MFCrossedProducts}, the crossed product $A\rtimes G$ is MF.
\end{proof}

\subsection{Quotients, Extensions, and Mayer-Vietoris.}
Arbitrary quotients of MF algebras need not be MF. Indeed, every separable, unital \cstar-algebra is a quotient of the residually finite dimensional algebra $\mathrm{C}^*(\mathbb{F}_{\infty})$. Similarly, extensions of MF algebras need not be MF. For example, the Toeplitz algebra $\mathcal{T}$ is an extension of $C(\mathbb{T})$ by the compact operators $\mathbb{K}$.  Both $\mathbb{K}$ and $C(\mathbb{T})$ are QD, hence MF, but $\mathcal{T}$ is not finite and so cannot be MF.  However, with a certain assumption on the extension, we show the MF property is a three-space property. We recall the following definition.

\begin{definition}\label{Quotients}
An exact sequence $0\rightarrow I\rightarrow A\rightarrow B\rightarrow0$ of C$^*$-algebras. The sequence is called \emph{quasidiagonal} if $I$ admits an approximate identity consisting of projections which is quasicentral in $A$.
\end{definition}

\begin{proposition}
If  the sequence $\xymatrix{0\ar[r]&I\ar[r]^\iota &A\ar[r]^\pi& B\ar[r]&0}$ is quasidiagonal and $A$ is MF, then $B$ is MF.
\end{proposition}

\begin{proof}
If $\tilde{\pi} : \tilde{A} \rightarrow \tilde{B}$ denotes the unitization of $\pi$, then $\tilde{A}$ is MF and
\[ \begin{tikzcd} 0 \arrow{r} & I \arrow{r}{\iota} & \tilde{A} \arrow{r}{\tilde{\pi}} & \tilde{B}  \arrow{r} & 0 \end{tikzcd} \]
is a quasidiagonal extension.  Hence we may assume $A$, $B$, and $\pi$ are unital.

Let $\varphi:B\rightarrow A$ be a linear splitting for the exact sequence; that is, $\varphi\circ\pi=\id_B$. We can always arrange $\varphi$ to be $\ast$-linear by replacing $\varphi$ with $1/2(\varphi+\varphi^*)$, where $\varphi^*(b)=\varphi(b^*)^*$.  Let $(p_n)_n$ be an approximate identity for $I$ which is quasicentral for $A$ such that each $p_n$ is a projection. Write $q_n=1_A-p_n$ and define
\[\varphi_n: B\rightarrow A,\qquad \varphi_n(b)=q_n\varphi(b)q_n.\]
Note that for $b\in B$, since $\varphi$ is a splitting, we have
\[\|b\|=\|\pi(\varphi(b))\|=\lim_{n\rightarrow\infty}\|q_n\varphi(b)q_n\|=\lim_{n\rightarrow\infty}\|\varphi_n(b)\|.\]
Since $(p_n)_n$ is quasicentral for $A$,
\[\|aq_n-q_na\|=\|p_na-ap_n\| \rightarrow 0 \]
as $n \rightarrow \infty$ for every $a \in A$. A little work shows that
\[\|\varphi_n(xy)-\varphi_n(x)\varphi_n(y)\|\stackrel{n\rightarrow\infty}{\longrightarrow}0\]
for $x,y\in B$.  Since $A$ is MF, it follows that $B$ is MF.
\end{proof}

Next, we consider extensions. Recall the following elementary fact: if $(T_n)_n \subset\mathbb{B}(\mathcal{H})$ is a net of operators converging to an operator $T$ in the strong operator topology, then $\|T\|\leq\liminf_{n\rightarrow\infty}\|T_n\|$. Moreover, if $\lim_{n\rightarrow\infty}\|T_n\|$ exists, then $\|T\|=\lim_{n\rightarrow\infty}\|T_n\|$.

\begin{proposition}\label{Extensions}
If $\xymatrix{0\ar[r]&I\ar[r]^\iota &A\ar[r]^\pi& B\ar[r]&0}$ is a quasidiagonal sequence with $I$ and $B$ MF algebras, then $A$ is also MF.
\end{proposition}

\begin{proof}
Identify $I$ with $\iota(I) $ and let $(p_n)_n$ be an approximate identity for $I$ which is quasicentral for $A$ such that each $p_n$ is a projection. The maps $\varphi_n:A\rightarrow I\oplus B$ given by $\varphi_n(a)=(p_nap_n,\pi(a))$ form an approximate morphism. Indeed,
\[\|p_nxyp_n-p_nxp_nyp_n\|\leq\|p_nx\|\|yp_n-p_ny\|\|p_n\|\leq\|x\|\|yp_n-p_ny\|\longrightarrow0.\]
Now we verify the approximate morphism $(\varphi_n)_n$ is faithful.  Set $\lambda=\liminf_{n\rightarrow\infty}\|p_nap_n\|$ and consider $I^{**}\subset A^{**}$. The fact that  $(p_n)_n$ is a quasicentral implies that the projections $p_n$ converge strongly to a central projection $P\in\mathcal{Z}(A^{**})$. Thus the map
\[A\hookrightarrow I^{**}\oplus B^{**}, \qquad a\mapsto (PaP, P^{\perp}aP^{\perp})\]
is a $\ast$-monomorphism which implies $\|a\|=\max\{\|PaP\|,\|P^{\perp}aP^{\perp}\|\}$. Using the remark stated before the proposition and the fact the the $(p_n)_n$ form an approximate unit for $I = \ker(\pi)$, we compute
\[\max\{\lambda, \|\pi(a)\|\}\geq\|\pi(a)\|=\lim_{n\rightarrow\infty}\|p_n^\perp a p_n^\perp \|=\|P^{\perp}aP^{\perp}\|,\]
and
\[\max\{\lambda, \|\pi(a)\|\}\geq \lambda\geq\|PaP\|.\]
Therefore, $\|a\|\leq\max\{\lambda, \|\pi(a)\|\}\leq\|a\|$, and the $\varphi_n$ are approximately isometric. Now $I\oplus B$ is MF, and $(\varphi_n)_n$ is a faithful approximate morphism from $A$ into $I \oplus B$.  Hence $A$ is MF by Proposition \ref{prop:permMF}.
\end{proof}

The MF property clearly passes to subalgebras, so a version of the three-space property now follows from Propositions~\ref{Quotients} and~\ref{Extensions}.

\begin{corollary}\label{ThreeSpace}
If  $\xymatrix{0\ar[r]&I\ar[r]^\iota &A\ar[r]^\pi& B\ar[r]&0}$ is a quasidiagonal sequence, then $A$ is MF if and only if $I$ and $B$ are MF.
\end{corollary}

The following proposition shows that under certain restrictions the MF property satisfies a Mayer-Vietoris-type condition.

\begin{proposition}\label{MV}
Let $A$ be a C*-algebra with ideals $I$ and $J$ such that $A=I+J$. Suppose the sequence $0\rightarrow I\cap J \rightarrow A\rightarrow A\big/(I\cap J)\rightarrow 0$ is quasidiagonal. If $I$ and $J$ are MF, then $A$ is also MF.
\end{proposition}

\begin{proof}
There is a quasidiagonal extension
\[ \begin{tikzcd} 0 \arrow{r} & I \cap J \arrow{r} & A \arrow{r}{\pi} & I / (I \cap J) \oplus J / (I \cap J) \arrow{r} & 0 \end{tikzcd} \]
where $\pi(x + y) = (x + I \cap J) \oplus (y + I \cap J)$ for all $x \in I$ and $y \in J$.
Clearly the sequences
\[0\longrightarrow I\cap J\longrightarrow I\longrightarrow I\big/(I\cap J)\longrightarrow 0\] and
\[0\longrightarrow I\cap J\longrightarrow J\longrightarrow J\big/(I\cap J)\longrightarrow 0\]
are quasidiagonal. Since $I$ and $J$ are MF, we have $I\cap J$ is MF, and hence Corollary~\ref{ThreeSpace} implies $I\big/(I\cap J)$ and $J\big/(I\cap J)$ are MF. Therefore, $(I\big/(I\cap J)) \oplus (J\big/(I\cap J))$ MF and again by Corollary~\ref{ThreeSpace}, $A$ is MF.
\end{proof}

\section{The Dynamics of Ordered Abelian Groups}\label{sec:KTheoreticDynamics}

Given a \cstar-dynamical system $(A,G,\alpha)$, there is an induced action $\mathrm{K}_0(\alpha):G \curvearrowright \mathrm{K}_0(A)$ given by ordered group automorphisms. The main idea of this section is to relate \cstar-dynamical properties with K-theoretic approximation properties. Both stable finiteness and the MF property of the crossed product are witnessed at the level of K-theory.  Also, if the acting group is free and if $A$ is classifiable, then the K-theoretic dynamics contain pertinent information about the crossed product. For example, in the case when $A$ is an AF-algebra and $G = \mathbb{Z}$, Brown has shown in \cite{BrownAFE} that the quasidiagonality of the crossed product $A \rtimes \mathbb{Z}$ is completely determined by the induced action of $\mathbb{Z}$ on $\mathrm{K}_0(A)$.  This was extended by the first author in \cite{Rainone} to the case where $A$ is AF and $G$ is free by introducing a K-theoretic notion of MF actions on ordered groups (Definition \ref{defn:K0MFaction} below).  Moreover, it was shown in \cite{Rainone} that an MF action of a group $G$ on an AF-algebra $A$ always induces an MF action of $G$ on $\mathrm{K}_0(A)$, and an action of $G$ on $\mathrm{K}_0(A)$ is MF if and only if the actions satisfies an analogue of Brown's coboundary condition from \cite{BrownAFE} (see Definition \ref{defn:coboundarysubgroup}).  We show below that both results hold without conditions on the algebra $A$.

\begin{definition}\label{defn:coboundarysubgroup}
Let $K = (K, K^+, u)$ be an ordered abelian group and consider an action $\sigma : G \curvearrowright K$.  Define the \emph{coboundary subgroup} of $\sigma$ to be the group
\[H_\sigma:=\big\langle x - \sigma_s(x)\ :\ x \in K,\ s\in G\big\rangle \leq K.\]
We say $\sigma$ satisfies the \emph{coboundary condition} if $H_\sigma \cap K^+ = \{0\}$.

If $A$ is a stably finite \cstar-algebra and $\alpha : G \curvearrowright A$ is an action, we say $\alpha$ satisfies the coboundary condition if $\mathrm{K}_0(\alpha)$ does. To ease the notation we write $H_\alpha$ in place of $H_{K_0(\alpha)}$.
\end{definition}

The stable finiteness of the crossed product and the coboundary condition are related.

\begin{proposition}\label{prop:StablyFiniteImpliesCoboundary} Let $\alpha:G\rightarrow\Aut(A)$ be an action.  If $A\rtimes_{\lambda}G$ is stably finite, then $\alpha$ satisfies the coboundary condition. The converse holds if $A$ is exact, $\mathrm{K}_0(A)$ separates traces on $A$, and $\alpha$ is minimal (i.e.\ has no invariant ideals).
\end{proposition}

\begin{proof}
When $A \rtimes_\lambda G$ is stably finite, the inclusion $\iota : A \rightarrow A \rtimes_\lambda G$ induces a faithful morphism $\mathrm{K}_0(\iota) : \mathrm{K}_0(A) \rightarrow \mathrm{K}_0(A \rtimes_\alpha G)$.  As $\iota \circ \alpha_s = \Ad_{u_s} \circ \iota$ for each $s \in G$, we have $\mathrm{K}_0(\iota) \circ \mathrm{K}_0(\alpha_s) = \mathrm{K}_0(\iota)$ for each $s \in G$.  Hence the coboundary subgroup $H_\alpha$ is contained in the kernel of $\mathrm{K}_0(\iota)$.  As $\iota$ is faithful, $H_\alpha \cap \mathrm{K}_0(A)^+ = \{0\}$ and hence $\alpha$ satisfies the coboundary condition.

Conversely, suppose $A$ is exact and has real rank zero and suppose $\alpha$ is minimal. Write $u=[1_A]_0$.  If $\alpha$ satisfies the coboundary condition, then $H_\alpha \cap \mathbb{Z}u = \{0\}$ and we may define a state $\beta_0 :H_\alpha + \mathbb{Z}u\rightarrow\mathbb{R}$ by $\beta_0(g+nu)=n$. By the Goodearl-Handelman Hahn-Banach Extension Theorem (see Theorem 6.8.3 in \cite{BlackadarKtheory}), $\beta_0$ extends to a state $\beta:\mathrm{K}_0(A)\rightarrow\mathbb{R}$. Note that $\beta$ is $\mathrm{K}_0(\alpha)$-invariant since $H_\alpha \subset \ker(\beta)$.  As $A$ is exact, $\beta=\mathrm{K}_0(\tau)$ for some trace $\tau$ on $A$ by combining the results of \cite{BlackadarRordam} and \cite{HaagerupQuasitrace} (see also Section 6.9 on \cite{BlackadarKtheory}).  As $\beta$ is $\mathrm{K}_0(\alpha)$-invariant, $\mathrm{K}_0(\tau \circ \alpha_s) = \mathrm{K}_0(\tau)$ for each $s \in G$, and, since $\mathrm{K}_0(A)$ separates traces on $A$, the trace $\tau$ is $\alpha$-invariant.  The ideal $N_\tau:=\{a\in A\ |\ \tau(a^*a)=0\}$ is an $\alpha$-invariant ideal, so minimality of the action implies that $N_\tau=\{0\}$ and hence $\tau$ is faithful. The composition $\tau\circ\mathbb{E}:A\rtimes_{\lambda}G\rightarrow\mathbb{C}$ is a faithful trace, and therefore $A\rtimes_{\lambda}G$ is stably finite.
\end{proof}

The following definition is taken from \cite{Rainone}.  It provides a notion of MF actions on ordered abelian groups.

\begin{definition}\label{defn:K0MFaction} Let $(K,K^+,u)$ be an ordered abelian group and consider an action $\sigma:G \rightarrow \Aut(K,K^+,u)$.  We say that $\sigma$ is MF if the following holds:

Given finite subsets $S \subset K^{+} \setminus \{0\}$ and $\Gamma \subset G$ there is a subgroup $H \leq K$, along with a group homomorphism $\mu: H \rightarrow \mathbb{Z}$ satisfying
\begin{enumerate}
\item $\{\sigma_{t}(g) : t \in \Gamma, \, g \in S \} \subset H$,
\item $\mu(g) > 0$ for each $g \in S$, and
\item $\mu(\sigma_{t}(g))=\mu(g)$ for all $g\in S$ and $t\in \Gamma$.
\end{enumerate}

If $(A,G,\alpha)$ is a C$^*$-dynamical system with $A$ stably finite, the action $\alpha$ is called $\mathrm{K}_0$-MF if the induced action $\mathrm{K}_0(\alpha)$ of $G$ on $\mathrm{K}_0(A)$ is MF.
\end{definition}

Note that when $K$ and $G$ are countable, an action $\sigma : G \curvearrowright K$ is MF if and only if there are $u_n \in \mathbb{Z}^+ \setminus \{0\}$ and a faithful approximate morphism
\[ (\lambda_n : K \rightarrow (\mathbb{Z}, \mathbb{Z}^+, u_n))_{n \geq 1} \]
which is constant on $G$-orbits in the following sense: if $x \in K$ and $s \in G$, then $\lambda_n(\sigma_s(x)) = \lambda_n(x)$ for all sufficiently large $n \in \mathbb{N}$.  Since there are faithful morphisms $(\mathbb{Z}, \mathbb{Z}^+, u_n) \rightarrow (\mathbb{Q}, \mathbb{Q}^+, 1)$ given by division by $u_n$, one has $\sigma$ is an MF action if and only if there is a faithful approximate morphism
\[ (\lambda_n : K \rightarrow (\mathbb{Q}, \mathbb{Q}^+, 1))_{n \geq 1} \]
which is constant on $G$-orbits.  Combining this observation with Proposition \ref{prop:groupapproxmorphism} yields the following result.

\begin{proposition}\label{prop:K0MFaction} Let $\sigma:G\rightarrow\Aut(K,K^+,u)$ be an action of a countable group $G$ on a countable ordered abelian group $K$.  Then $\sigma$ is MF if and only if there is a faithful morphism $\mu:K \rightarrow \mathbb{Q}_\omega$ with $\mu\circ\sigma_s = \mu$ for every $s \in G$.
\end{proposition}

The following result was proven in the case of AF-algebras in \cite{Rainone}.

\begin{theorem}\label{thm:MFImpliesK0MF}
Suppose $A$ is a unital C*-algebra and $\alpha : G \rightarrow\Aut(A)$ is a given action.  If  $\alpha$ is MF, then $\alpha$ is K$_0$-MF.
\end{theorem}

\begin{proof}
We may assume $A$ is separable and $G$ is countable.  There is a unitary representation $w:G \rightarrow \mathcal{U}(\mathcal{Q}_\omega)$ and a faithful morphism $\varphi:A \rightarrow \mathcal{Q}_\omega$ such that $\varphi \circ \alpha_s = \Ad_{w_s}\circ\varphi$.  Since $\mathcal{Q}_\omega$ is stably finite, the induced map $\mathrm{K}_0(\varphi) : \mathrm{K}_0(A) \rightarrow \mathrm{K}_0(\mathcal{Q}_\omega)$ is faithful.  Moreover, for $s \in G$,
\[\mathrm{K}_0(\varphi)\circ \mathrm{K}_0(\alpha_s) = \mathrm{K}_0(\varphi \circ \alpha_s) = \mathrm{K}_0(\Ad_{w_s} \circ \varphi_s) =\mathrm{K}_0(\varphi_s).\]
By Theorem \ref{prop:K0MFaction} and Proposition \ref{prop:K0UHFUltrapower}, $\mathrm{K}_0(\alpha)$ is MF as required.
\end{proof}

We now embark on showing that the MF property for actions and the coboundary condition are equivalent in full generality. We do this by first showing that any ordered abelian group exhibits an MF property. As with C$^*$-algebras, one could define the matricial field property for countable ordered abelian groups as follows:  $K$ if MF if and only if there is a faithful morphism $K \rightarrow \mathrm{K}_0(\mathcal{Q})_\omega\cong\mathbb{Q}_{\omega}$; in other words, the action of the trivial group of $K$ is MF.  The following result shows the matricial field property for ordered abelian groups is automatic.

\begin{theorem}\label{thm:K0QDQ}
For any countable ordered abelian group $(K,K^+,u)$, there is a faithful morphism $\lambda: K\rightarrow\mathrm{K}_0(\mathcal{Q}_\omega) \cong \mathbb{Q}_\omega$.
\end{theorem}

\begin{proof}
Enumerate the positive cone $K^+=\{0, x_0=u, x_1, x_2,\dots\}$ and fix $n\in\mathbb{N}$. The subgroup $H_n:=\langle x_0,x_1,\dots,x_n\rangle\leq K$ is finitely generated and abelian and thus decomposes as $H_n\cong F\oplus T$, where $T$ a finite group and $F$ is a free abelian group. Fix an isomorphism $\phi:F\rightarrow\mathbb{Z}^d$ of (unordered) groups. For each $0\leq i\leq n$, write $x_i = y_i\oplus z_i$ according to this decomposition.  Consider the convex hull
\[C:=\conv\{\phi(y_i)\ |\ 0\leq i\leq n\}\subset\mathbb{R}^d.\]
Note that $C$ is a compact, convex set in $\mathbb{R}^d$.  We claim $C$ does not contain the origin $0_{\mathbb{R}^d}$. By way of contradiction, suppose $\sum_{i=0}^{n}t_i\phi(y_i)=0_{\mathbb{R}^d}$, with $t_i\geq0$ and $\sum_{i=0}^{n}t_i=1$. Since $\phi(y_i)\in\mathbb{Z}^d$, we may assume that $t_i=a_i/b$ where $a_i\in\mathbb{Z}^+$ and $b\in\mathbb{N}$. Then
\[\phi\left(\sum_{i=0}^n a_iy_i\right)=\sum_{i=0}^n a_i\phi(y_i)=b\sum_{i=0}^n \frac{a_i}{b}\phi(y_i)=b\sum_{i=0}^n t_i\phi(y_i)=0_{\mathbb{R}^d}.\]
Since $\phi$ is injective we have $\sum_{i=0}^n a_iy_i=0_{F}$.  Let $m$ denote the order of $T$ and fix $i_0$ such that $a_{i_0} \neq 0$.  Then
\[ 0\lneqq a_{i_0}mx_i=a_{i_0}(my_i+mz_i)=a_{i_0}my_i \leq m \sum_{i=0}^n a_iy_i=0_{F}, \]
a contradiction.  So the claim holds.

There is, therefore, a linear functional $f : \mathbb{R}^d \rightarrow \mathbb{R}$ with $f(C) \subseteq (\varepsilon, \infty)$ for some $\varepsilon>0$. Denoting the standard basis of $\mathbb{R}^d$ by $\{e_i\}_{i=1}^{d}$, we may perturb $f$ so that $f(e_i) \in \mathbb{Q}$ and $f(C)\subset (\varepsilon/2, \infty)$.  Let $\tilde{\lambda}_n:H_n \rightarrow \mathbb{R}$ denote the composition
\[ H_n \overset{\pi_F}{\twoheadrightarrow }F \overset{\phi}{\longrightarrow} \mathbb{Z}^d \overset{\iota}{\hookrightarrow}\mathbb{R}^d\overset{f}{\longrightarrow} \mathbb{R}. \]
Then $\tilde{\lambda_n}$ is $\mathbb{Q}$-valued. Now define $\lambda_n:H_n\rightarrow\mathbb{Q}$ by $\lambda_n:=\tilde{\lambda}_n(\cdot)/\tilde{\lambda}_n(u)$ and extend to a function $\lambda_n: K\rightarrow\mathbb{Q}$ arbitrarily. It is a simple matter to check that the sequence $(\lambda_n:K\rightarrow \mathbb{Q})_{n\geq1}$ is a faithful approximate morphism of ordered abelian groups. Proposition \ref{prop:groupapproxmorphism} furnishes the desired faithful morphism $\lambda:K\rightarrow\mathbb{Q}_{\omega}$.
\end{proof}

The equivariant version of Theorem \ref{thm:K0QDQ} is false; that is, not every action on an ordered abelian group is MF.  For example, if $\alpha$ is an automorphism of an ordered abelian group $K$ and $\alpha(x) < x$ for some $x \in K$, then the action of $\mathbb{Z}$ induced by $\alpha$ is not MF.  The coboundary condition gives a more general obstruction to an action being MF.  In fact, the theorem below shows the coboundary condition is the only obstruction.  This generalizes Proposition 4.13 of \cite{Rainone} where the result is shown for dimension groups.

\begin{theorem}\label{thm:K0EquivariantQDQ}
Let $\sigma: G\rightarrow\Aut(K,K^+,u)$ be an action. Then $\sigma$ is MF if and only if it satisfies the coboundary condition.
\end{theorem}

\begin{proof}
We may assume $K$ and $G$ are countable.  If $\sigma$ is MF, there is a faithful morphism $\mu:K\rightarrow\mathbb{Q}_\omega$ which is $\sigma$-invariant. Therefore, if $x\in K$ and $t\in G$, then $\mu(x-\sigma_{t}(x))=0$ which implies that $H_{\sigma}\subset\ker(\mu)$. Since $\mu$ is faithful, $H_\sigma$ cannot contain a non-zero positive element.

Now assume that $\sigma$ satsifies the coboundary condition. Define $L=K/H_\sigma$ to be the quotient group and set $L^+:=\pi(K^+)$ and $w=\pi(u)$ where $\pi:K\rightarrow L$ is the quotient map. Using the coboundary condition, it is easily verified that $(L,L^+,w)$ is an ordered abelian group and that $\pi$ is faithful. Theorem \ref{thm:K0QDQ} yields a faithful morphism $\lambda : L \rightarrow \mathbb{Q}_\omega$, and the composition $\mu:=\lambda\circ\pi$ is the map required to conclude that $\sigma$ is MF using Proposition \ref{prop:K0MFaction}.
\end{proof}

We end this section by presenting a tracial version of Theorem~\ref{thm:K0QDQ}, but, first, we fix some notation. Let $\tau_{\mathcal{Q}}$ denote the unique trace on $\mathcal{Q}$ and let $\tau_{\mathcal{Q}_{\omega}}$ be the induced trace on $\mathcal{Q}_\omega$. Note that $\mathrm{K}_0(\tau_{\mathcal{Q}})$ is an isomorphism $K_0(\tau_{\mathcal{Q}}) \overset{\cong}{\longrightarrow} \mathbb{Q} \subseteq \mathbb{R}$. Also, Proposition~\ref{prop:K0UHFUltrapower} states $K_0(\mathcal{Q}_\omega)\cong\mathbb{Q}_\omega$, and the state $\iota_\omega$ on $\mathbb{Q}_\omega$ induced by the state on $K_0(\tau_{\mathcal{Q}_\omega})$ can be described explicitly as $\displaystyle \iota_\omega(\kappa_\omega((q_n)_n)) = \lim_{n \rightarrow \omega} q_n$, for $(q_n)_n\in\ell^\infty(\mathbb{Q})$ . We will always view $\mathbb{Q}_\omega$ as being equipped with this state.  In fact, this is probably the only state.

\begin{theorem}\label{thm:K0QDTQ}
Suppose $K$ is a countable ordered abelian group, $G$ is a countable group, $\sigma : G \curvearrowright K$ is an action, and $\beta$ is a $\sigma$-invariant state on $K$. Then there is a state-preserving morphism $\lambda:K \rightarrow\mathbb{Q}_\omega$ such that $\lambda\circ \sigma_s = \lambda$ for every $s\in G$.
\end{theorem}

\begin{proof}
Let $\hat{K}$ denote the quotient $K / \ker(\beta)$ and let $\pi : K \rightarrow \hat{K}$ denote the quotient map.  Then $\hat{K}$ is an ordered abelian group with positive cone $\hat{K}^+ = \pi(K^+)$ and order unit $\hat{u} = \pi(u)$.  The quotient map $\pi$ is constant on the orbits of $\sigma$ since $\beta$ is $\sigma$-invariant.  Also, the state $\beta$ on $\hat{K}$ induces a faithful state $\hat{\beta} : \hat{K} \rightarrow \mathbb{R}$.  To complete the proof, it suffices to produce a state-preserving morphism $\hat{K} \rightarrow \mathbb{Q}_\omega$.

Note that $\hat{\beta} : \hat{K} \rightarrow \mathbb{R}$ is injective and hence $\hat{K}$ is torsion free.  Fix a finite set $S \subseteq \hat{K}^+ \setminus \{0\}$ with $\hat{u} \in S$ and fix $\varepsilon > 0$.  Decreasing $\varepsilon$, we may assume $\varepsilon < \hat{\beta}(x)$ for each $x \in S$.  Let $H$ denote the subgroup of $\hat{K}$ generated by $S$ and fix a basis $e_1, \ldots, e_n$ for $H$.  Given $\delta > 0$ (to be determined later), there are $t_1, \ldots, t_n \in \mathbb{Q}$ with $t_i \approx_\delta \hat{\beta}(e_i)$.  Define $\lambda_0 : H \rightarrow \mathbb{Q}$ by $\lambda_0(e_i) = t_i$.  As each $x \in S$ is a $\mathbb{Z}$-linear combination of the $e_i$ and $\lambda_0(e_i) \approx_\delta \hat{\beta}(e_i)$ for each $i =1, \ldots, n$, we have $\lambda_0(x) \approx_\varepsilon \hat{\beta}(x)$ for all $x \in S$ so long as $\delta > 0$ is sufficiently small.  In particular, $\lambda_0(x) > 0$ for all $x \in S$.

For every $S \subseteq \hat{K}^+$ and $\varepsilon > 0$, we have constructed $\lambda_0 : \langle S \rangle \rightarrow \mathbb{Q}$ with $\lambda_0(x) > 0$ and $\lambda_0(x) \approx_\varepsilon \hat{\beta}(x)$ for all $x \in S$; note in particular that $\lambda_0(\hat{u}) \approx_\varepsilon 1$.  From here, constructing a state-preserving morphism $\hat{K} \rightarrow \mathbb{Q}_\omega$ is routine using Proposition \ref{prop:groupapproxmorphism}.  As mentioned in the first paragraph of the proof, composing this morphism with the quotient map $\pi : K \rightarrow \hat{K}$ yields the desired morphism $\lambda : K \rightarrow \mathbb{Q}_\omega$.
\end{proof}

\section{Some Consequences of Classification}\label{sec:ClassificationStuff}

The results in the previous section essentially solve the question when an action is MF at the level of the K$_0$-group.  With restrictions on the class of C$^*$-algebras considered, classification techniques allow for the K$_0$-approximations to be lifted to norm approximations and hence one can recover MF-approximations of dynamical systems from the K-theory.  This section is devoted to collecting the classification results and their consequences needed for our main results.

\begin{theorem}\label{thm:LiftingK0Approximations}
Suppose $A$ is a unital AF-algebra and $\lambda: \mathrm{K}_0(A) \rightarrow \mathrm{K}_0(\mathcal{Q}_\omega)$ is a morphism of ordered abelian groups.  There is a unital, $\ast$-homomorphism $\varphi:A\rightarrow \mathcal{Q}_\omega$ with $\mathrm{K}_0(\varphi) = \lambda$.  If $\psi$ is another such morphism, then $\varphi$ and $\psi$ are unitarily equivalent.
\end{theorem}

\begin{proof}
Note that $\mathcal{Q}_\omega$ has cancellation of projections and hence for a finite dimensional C$^*$-algebra $B$, any morphism $\mathrm{K}_0(B) \rightarrow \mathrm{K}_0(\mathcal{Q}_\omega)$ lifts to a morphism $B \rightarrow \mathcal{Q}_\omega$ which is unique up to unitary equivalence (see Lemma 7.3.2 in \cite{RordamBlueBook}).  By Elliott's one-sided intertwining argument, if $A$ is an AF-algebra, every morphism $\mathrm{K}_0(A) \rightarrow \mathrm{K}_0(\mathcal{Q}_\omega)$ lifts to a morphism $A \rightarrow \mathcal{Q}_\omega$ which is unique up to approximate unitary equivalence.  A standard reindexing argument shows two morphisms form $A$ into $\mathcal{Q}_\omega$ are approximately unitarily equivalent if and only if they are unitarily equivalent and this completes the proof.
\end{proof}

Recall that a C$^*$-algebra has \emph{real rank zero} if every self-adjoint element can be approximated arbitrarily well by a self-adjoint with finite spectrum. Also, an A$\mathbb{T}$-algebra is a C$^*$-algebra which can be expressed as an inductive limit of C$^*$-algebras of the form $C(\mathbb{T})\otimes F$ where $F$ is finite dimensional.  A remarkable theorem of Elliott (see \cite{ElliottATAlgebras}) gives a complete classification of A$\mathbb{T}$-algebras of real rank zero in terms of their K-theory.  We reproduce his result in Theorem \ref{thm:ElliottATAlgebras}.  First we recall the invariant.

For a C$^*$-algebra $A$, let $\mathrm{K}_*(A) := \mathrm{K}_0(A) \oplus \mathrm{K}_1(A)$ which is viewed as a $\mathbb{Z}/2$-graded abelian group.  Define the \emph{graded dimension range} by
\[\mathcal{D}(A) =\big \{([p]_0, [v+p^\perp ]_1)\ |\  p\in\mathcal{P}(A), v\in\mathcal{U}(pAp)\big \}\subset \mathrm{K}_*(A),\]
and let $\mathrm{K}_*^+(A)$ denote the semigroup in $\mathrm{K}_*(A)$ generated by $\mathcal{D}(A)$.  Then $(\mathrm{K}_*(A), \mathrm{K}_*^+(A))$ is an ordered abelian group. A morphism $\theta:(\mathrm{K}_*(A), \mathcal{D}(A))\rightarrow(\mathrm{K}_*(B),\mathcal{D}(B)) $ is a group homomorphism $\theta:\mathrm{K}_*(A)\rightarrow\mathrm{K}_*(B)$ satisfying
\[ \theta(\mathrm{K}_0(A)) \subseteq \mathrm{K}_0(B), \quad \theta(\mathrm{K}_1(A)) \subseteq \mathrm{K}_1(B), \quad \text{and} \quad \theta(\mathcal{D}(A)) \subseteq \mathcal{D}(B). \]
Every $\ast$-homomorphism $\varphi:A\rightarrow B$ between unital \cstar-algebras induces a morphism $\mathrm{K}_*(\varphi):(\mathrm{K}_*(A), \mathcal{D}(A))\rightarrow(\mathrm{K}_*(B),\mathcal{D}(B)) $ in a natural way.

\begin{theorem}[Elliott, \cite{ElliottATAlgebras}]\label{thm:ElliottATAlgebras}
Let $A$ and $B$ be A$\mathbb{T}$-algebras of real rank zero.  Given a  morphism $\theta:(\mathrm{K}_*(A),\mathcal{D}(A) \rightarrow (\mathrm{K}_*(B),\mathcal{D}(B))$, there is a $\ast$-homomorphism $\varphi: A \rightarrow B$ with $\mathrm{K}_*(\varphi)=\theta$.  Moreover, any two such lifts are approximately unitarily equivalent.
\end{theorem}

The following simple consequence of Elliott's classification of A$\mathbb{T}$-algebras is vital in proving our main result, indeed, it will allow us to reduce from A$\mathbb{T}$-algebras of real rank zero to AF-algebras.

\begin{corollary}\label{cor:K0EmeddingOfATAlgebras}
If $A$ is an A$\mathbb{T}$-algebra of real rank zero, there is an AF-algebra $B$ and an embedding $\varphi:A \hookrightarrow B$ such that $\mathrm{K}_0(\varphi)$ is an isomorphism of ordered abelian groups.
\end{corollary}

\begin{proof}
It is well-known that $\mathrm{K}_0(C(\mathbb{T}))\cong\mathbb{Z}$, so $\mathrm{K}_0(C(\mathbb{T}, F)) \cong \mathrm{K}_0(F)$ as ordered abelian groups for any finite dimensional C$^*$-algebra $F$.  By the continuity of K-theory, $\mathrm{K}_0(A)$ is a dimension group.  Therefore, there is an AF-algebra $B$ and an isomorphism of ordered abelian groups $\theta : \mathrm{K}_0(A) \rightarrow \mathrm{K}_0(B)$.  As $B$ is an AF-algebra, $B$ is also an A$\mathbb{T}$-algebra of real rank zero, and, by Elliott's Theorem (Theorem \ref{thm:ElliottATAlgebras}), there is a $\ast$-homomorphism $\varphi : A \rightarrow B$ with $\mathrm{K}_0(\varphi) = \theta$. It only remains to show $\varphi$ is injective.  If $p$ is a projection in $\ker(\varphi)$, then $\theta([p]_0) = [\varphi(p)]_0 =0$.  As $\theta$ is injective, $[p]_0 = 0$, and as $A$ is stably finite, $p = 0$.  Since $A$ has real rank zero, $\ker(\varphi)$ also has real rank zero and hence is the closed span of its projections. Thus $\ker(\varphi) =\{0\}$.
\end{proof}

Recall that a C$^*$-algebra $A$ is \emph{subhomogeneous} if there is an $n \in \mathbb{N}$ such that every irreducible representation $\pi : A \rightarrow \mathbb{B}(\mathcal{H})$  satisfies $\dim(\mathcal{H})\leq n$.  A C$^*$-algebra is \emph{approximately subhomogeneous} (ASH) if it is a direct limit of subhomogeneous C$^*$-algebras. Using a deep classification result of Dadarlat and Gong in \cite{DadarlatGong}, we can further reduce the approximation problem from ASH-algebras of real rank zero to A$\mathbb{T}$-algebras of real rank zero.

\begin{theorem}\label{thm:ASHimpliesRationallyAT}
If $A$ is an ASH-algebra of real rank zero, then $A \otimes \mathcal{Q}$ is an A$\mathbb{T}$-algebra of real rank zero.
\end{theorem}

Although Theorem \ref{thm:ASHimpliesRationallyAT} appears to be well known to those working in classification, the result is not easy to extract from the literature.  The rest of this section is devoted to showing how Theorem \ref{thm:ASHimpliesRationallyAT} follows from classification results with the key tools being the classification of real rank zero ASH-algebras due to Dadarlat and Gong in \cite{DadarlatGong} and the range of the invariant for real rank zero A$\mathbb{T}$-algebras described by Elliott in \cite{ElliottATAlgebras}.

It is a classical result that $A$ is subhomogeneous if and only if there is an embedding of $A$ into $\mathbb{M}_n(C(X))$ for some $n \in \mathbb{N}$ and some compact metric space $X$.  In the real rank zero case, with a mild assumption (which will be automatic for $A \otimes \mathcal{Q}$), a theorem of Dadarlat and Gong gives a much more refined result on the structure of ASH-algebras (Theorem \ref{ASHBuildingBlocks} below).  Before stating this result, we need some notation.

Define the \emph{dimension drop algebra} to be
\[ \tilde{\mathbb{I}}_n := \{ a \in C([0, 1], \mathbb{M}_n) : a(0), a(1) \in \mathbb{C} 1_{\mathbb{M}_n} \}. \]
Every ASH-algebra can be written as an inductive system $(A_n, \nu_n)$ where each $A_n$ has the form
\[ A_n = \left( \bigoplus_{i=1}^r \mathbb{M}_{k_i}(C(X_i)) \right) \oplus \left(\bigoplus_{j=1}^s \mathbb{M}_{l_j}(\tilde{\mathbb{I}}_{m_j}) \right) \]
where each $X_i$ is a finite CW complex.

The following technical condition is Definition 2.1 in \cite{DadarlatGong}.  It is a necessary condition for classification, but it can always be arranged to hold in our setting by tensoring with $\mathcal{Q}$.

\begin{definition}\label{defn:slowdimensiongrowth}
An ASH-algebra $A$ has \emph{slow dimension growth} if $A$ is isomorphic to the limit of an inductive system as described above such that
\[ \inf_{n \in \mathbb{N}} \, \sup_{d \in \mathbb{N}} \, \lim_{r \rightarrow \infty} \, \inf \, \left\{ \frac{\operatorname{rank} \nu^{j, i}_{n, r}(1_{A_{j, i}})}{\operatorname{dim}(X_{j, r})} : \operatorname{dim}(X_{j, r}) \geq d, \, \operatorname{rank} \nu^{j, i}_{r, n}(1_{A_n}) \neq 0 \right\} = + \infty, \]
where the $\nu^{j, i}_{r, n}$ denote the components of the maps $\nu_{r, n} : A_n \rightarrow A_r$.
\end{definition}

For us, the most important point to note about the definition is that if $B$ is an infinite dimensional UHF-algebra and $A \otimes B \cong A$, then $A$ has slow dimension growth since by replacing the $A_n$ with $\mathbb{M}_{k_n}(A_n)$ for integers $k_n$ increasing sufficiently rapidly, one may arrange for the numerators in the expression above to be arbitrarily large.

The following deep structural result is Theorem 8.14 in \cite{DadarlatGong}.  We write $\mathbb{D}$ to denote the closed unit disc in $\mathbb{C}$ and write $\mathbb{W}_n$ to denote the CW complex obtained as the pushout of the inclusion $\mathbb{T} \rightarrow \mathbb{D}$ and the map $\mathbb{T} \rightarrow \mathbb{T}$ given by $z \mapsto z^n$.

\begin{theorem}[Dadarlat and Gong]\label{ASHBuildingBlocks}
If $A$ is an ASH-algebra of real rank zero with slow dimension growth, then $A$ is a direct limit of finite direct sums of matrix algebras over the C*-algebras $\mathbb{C}$, $C(\mathbb{T})$, $C(\mathbb{W}_n)$, and $\tilde{\mathbb{I}}_n$.
\end{theorem}

It is known that each of the four algebras above has cancellation of projections.  In particular, every real rank zero ASH-algebra with slow dimension growth has cancellation of projections and hence also has stable rank one.

There is a complete classification of real rank zero ASH-algebras with slow dimension growth, also due to Dadarlat and Gong in \cite{DadarlatGong}.  The invariant is the scaled, ordered, total K-theory $(\underline{\mathrm{K}}(A), \underline{\mathrm{K}}(A)^+, \Sigma(A))$ where $\underline{\mathrm{K}}(A) = \bigoplus_{n=0}^\infty \mathrm{K}_*(A ; \mathbb{Z}/n)$ equipped with the action of the Bockstein operators on $\underline{\mathrm{K}}(A)$ (the precise definitions can be found in Section 4 of \cite{DadarlatGong} but will not be relevant for us).  The groups $\mathrm{K}_*(A ; \mathbb{Z} / n)$ fit into an exact sequence
\[ \mathrm{K}_i(A) \overset{n}{\longrightarrow} \mathrm{K}_i(A) \rightarrow \mathrm{K}_i(A ; \mathbb{Z}/n) \rightarrow \mathrm{K}_{1-i}(A) \overset{n}{\longrightarrow} \mathrm{K}_{1-i}(A)\quad i=1,2 \]
where the $n$ denotes the map defined by multiplication by $n$.  If $A \cong A \otimes \mathcal{Q}$, then $\mathrm{K}_*(A) \cong \mathrm{K}_*(A) \otimes \mathbb{Q}$ and, in particular, $\mathrm{K}_*(A)$ is divisible. Hence multiplication by $n$ is an automorphism of $\mathrm{K}_*(A)$ for $n > 0$ and $\mathrm{K}_*(A ; \mathbb{Z}/n) = 0$ for $n > 0$.  Also one can show the Bockstein operations are trivial in this case since $\mathrm{K}_*(A ; \mathbb{Z}/n) = 0$ for $n > 0$.  Hence when $A \cong A \otimes \mathcal{Q}$,
\[ \underline{\mathrm{K}}(A) = \mathrm{K}_*(A) = \mathrm{K}_0(A) \oplus \mathrm{K}_1(A). \]

We describe the order and scale on $\mathrm{K}_*(A)$ in Dadarlat and Gong's classification theorem.  The scale is given by $\Sigma(A) := \{ [p]_0 \in \mathrm{K}_0(A)\ |\ p \in \mathcal{P}(A) \}$, and the order structure is induced by an isomorphism $\mathrm{K}_0(C(\mathbb{T},A)) \rightarrow \mathrm{K}_*(A)$
which we briefly describe below.

Consider the split exact sequence
\[ \begin{tikzcd} 0 \arrow{r} & SA \arrow{r}{\iota} & C(\mathbb{T}, A) \arrow[yshift = -.5 ex]{r}[swap]{\varepsilon} & A \arrow{r} \arrow[yshift = .5 ex]{l}[swap]{j} & 0 \end{tikzcd} \]
where $\varepsilon(f)=f(1)$, $\iota$ is inclusion, and $j(a)(z)=a$, $z\in\mathbb{T}, a\in A$. This induces a split exact sequence at the level of $\mathrm{K}_0$ and an isomorphism $\varphi_A : \mathrm{K}_0(C(\mathbb{T}, A)) \rightarrow \mathrm{K}_0(A) \oplus \mathrm{K}_0(SA)$ given by
\[ \varphi_A([p]_0) = [p(1)]_0 \oplus ([p]_0 - [p(1)]_0),\quad p\in\mathcal{P}_\infty(C(\mathbb{T},A)).\]
Now recall that there is an isomorphism $\theta_A : \mathrm{K}_1(A) \rightarrow \mathrm{K}_0(SA)$, where $SA = C_0(\mathbb{R}) \otimes A$ is the suspension of $A$ (see, for example, Section 10.1 of \cite{RordamBlueBook}).  The map $\theta_A$ is defined as follows: given $u \in \mathcal{U}_n(A)$, the unitary $u \oplus u^* \in \mathcal{U}_{2n}(A)$ is homotopic to the identity; that is, there is a continuous path $w : [0, 1] \rightarrow \mathcal{U}_{2n}(A)$ such that $w(0) = u \oplus u^*$ and $w(1) = 1_{2n}$.  Define a path of projections $q : [0, 1] \rightarrow \mathcal{P}_{2n}(A)$ by $q(t) = w(t) (1_n \oplus 0_n) w(t)^*$.  Then $q(0) = q(1) = 1_n \oplus 0_n$  and $[q]_0 - n [1]_0 \in \mathrm{K}_0(SA)$.  The map $\theta_A : \mathrm{K}_1(A) \rightarrow \mathrm{K}_0(SA)$ given by $\theta_A([u]_1) = [q]_0 - n[1]_0$ is an isomorphism. By composing $\varphi_A$ with $\operatorname{id} \oplus \theta_A^{-1}$, one obtains an isomorphism
\[ \psi_A : \mathrm{K}_0(C(\mathbb{T}, A)) \overset{\cong}\longrightarrow \mathrm{K}_0(A) \oplus \mathrm{K}_1(A)=\mathrm{K}_*(A).\]
Define $\mathrm{K}_*(A)^+ \subseteq \mathrm{K}_*(A)$ to be the image of $\mathrm{K}_0(C(\mathbb{T}, A))^+$ under this isomorphism.

We will show that for real rank zero ASH-algebras $A$ with slow dimension growth $\mathrm{K}_*(A)^+ = \mathrm{K}_*^+(A)$. One inclusion holds in full generality.

\begin{proposition}\label{prop:ConeContainment}
If $A$ is a unital C*-algebra, then $\mathrm{K}_*^+(A) \subseteq \mathrm{K}_*(A)^+$.
\end{proposition}

\begin{proof}
It suffices to show that $\mathcal{D}(A) \subseteq \mathrm{K}_*(A)^+$. To this end, fix a projection $p \in A$ and a unitary $v\in \mathcal{U}(pAp)$, and set $u=v+(1-p)$. The unitary $v \oplus v^*$ in $\mathcal{U}_2(pAp)$ is homotopic to the identity $p\oplus p$.  Let $w : [0, 1] \rightarrow \mathcal{U}_2(pAp)$ be a continuous path such that $w(0) = v \oplus v^*$ and $w(1) = p \oplus p$ and set $\tilde{w}(t) = w(t) + (1 - p) \oplus (1 - p)$. Then $\tilde{w}:[0,1]\rightarrow\mathcal{U}_2(A)$ is a continuous path of unitaries with $\tilde{w}(0)=u\oplus u^*$ and $\tilde{w}(1)=1 \oplus 1$. We define continuous paths of projections $q, \tilde{q}:[0,1]\rightarrow\mathcal{P}_2(A)$ given by $q(t) = w(t) (p\oplus 0) w(t)^*$, and $\tilde{q}(t) = \tilde{w}(t) (1 \oplus 0) \tilde{w}(t)^*$ for $t \in [0, 1]$.  Note that $\tilde{q}(t)=q(t)+p^\perp\oplus 0$. Now $q(0) = q(1) = p \oplus 0$ and hence $q$ furnishes an element $[q]_0 \in \mathrm{K}_0(C(\mathbb{T}, A))^+$.  Also, $[q(1)]_0 = [p]_0$ in $\mathrm{K}_0(A)$,  and $\theta([u]_1)=[\tilde{q}]_0-[\tilde{q}(1)]_0=[q]_0-[q(1)]_0$ in $\mathrm{K}_0(SA)$.  Therefore, $\psi_A([q]_0) = [p]_0 \oplus [u]_1$ and $[p]_0 \oplus [u]_1 \in \mathrm{K}_*(A)^+$ as claimed.
\end{proof}

\begin{proposition}
If $A$ is an ASH-algebra of real rank zero and slow dimension growth, then $\mathrm{K}_*(A)^+ = \mathrm{K}_*^+(A)$.
\end{proposition}

\begin{proof}
For any algebra $B$, it follows from definitions that
\[ \mathrm{K}_*(B)^+ \subseteq \{x \oplus y \in \mathrm{K}_*(B) : x > 0 \} \cup \{(0, 0)\}=:S_B.\]
Examining the building blocks of $A$ more carefully we have
\[ \begin{array}{ccccccc} \mathrm{K}_0(\mathbb{C}) = \mathbb{Z}, & & \mathrm{K}_0(C(\mathbb{T})) = \mathbb{Z}, & & \mathrm{K}_0(C(\mathbb{W}_n)) = \mathbb{Z} \oplus \mathbb{Z} / n, & & \mathrm{K}_0(\tilde{\mathbb{I}}_n) = \mathbb{Z}, \\ \mathrm{K}_1(\mathbb{C}) = 0, & & \mathrm{K}_1(C(\mathbb{T})) = \mathbb{Z}, & & \mathrm{K}_1(C(\mathbb{W}_n)) = 0, & & \mathrm{K}_1(\tilde{\mathbb{I}}_n) = \mathbb{Z}/n. \end{array} \]
For $B = \mathbb{C}$ and $B= C(\mathbb{W}_n)$, it is clear that $\mathrm{K}_*^+(B) = S_B$.  For $B= C(\mathbb{T})$ or $B = \tilde{\mathbb{I}}_n$, Paragraph 4.1 in \cite{ElliottATAlgebras} shows
$\mathrm{K}_*^+(B) = S_B$. These observations together with Proposition~\ref{prop:ConeContainment} shows that $\mathrm{K}_*(B)^+ = \mathrm{K}_*^+(B)$ for each building block $B$.

Now observe that $\mathrm{K}_*(-)^+$ preserves direct sums and direct limits by the naturality of the maps $\varphi_A$ and $\theta_A$ and the fact that both properties are true for $\mathrm{K}_0(C(\mathbb{T}, -))^+$.  It is also true that the positive cone $\mathrm{K}_*^+(-)$ defined above preserves direct sums and direct limits. The result follows by Theorem \ref{ASHBuildingBlocks}.
\end{proof}

\begin{proof}[Proof of Theorem \ref{thm:ASHimpliesRationallyAT}]
Let $A$ be an ASH-algebra of real rank zero.  We may assume $A$ is unital and $A \otimes \mathcal{Q} \cong A$.  Then $A$ has slow dimension growth as noted after Definition \ref{defn:slowdimensiongrowth} and hence Theorem \ref{ASHBuildingBlocks} implies $A \cong \lim_{n\rightarrow\infty}(A_n, \varphi_n)$ where each $A_n$ is a finite direct sum of matrix algebras over $\mathbb{C}$, $C(\mathbb{T})$, $C(\mathbb{W}_m)$, and $\tilde{\mathbb{I}}_m$.  Let $L_n = \mathrm{K}_*(A_n)$ and $L = \mathrm{K}_*(A)$ as graded, ordered abelian groups. We can write $L_n=F_n\oplus T_n$ where $F_n$ is free abelian and $T_n$ is a finite abelian group.  Endowing $F_n$ with the order and grading inherited from $L_n$, the projection map $\pi_n : L_n \twoheadrightarrow F_n$ is a positive, graded morphism as can be seen by examining the K-theory of each building block.  Also, if $x \in F_n$ is positive, then $x \oplus 0 \in L_n$ is positive and hence $\pi$ restricts to a surjective map on the positive cones $L_n^+ \rightarrow F_n^+$.  Each $\mathrm{K}_*(\varphi_n)$ decomposes as a matrix
\[ \mathrm{K}_*(\varphi_n) = \begin{pmatrix} \alpha_n & 0 \\ \beta_n & \gamma_n \end{pmatrix} : F_n \oplus T_n \rightarrow F_{n+1} \oplus T_{n+1}, \]
where the (1, 2)-entry is zero as there are no non-trivial morphisms from a torsion group to a free group.

If $F = \underset{\longrightarrow}{\lim} \, (F_n, \alpha_n)$, the maps $\pi_n : L_n \rightarrow F_n$ induce a map $\pi : L \rightarrow F$ as shown below:
\[ \begin{tikzcd} L_n \arrow{r}{\mathrm{K}_*(\varphi_n)} \arrow{d}{\pi_n} & L_{n+1} \arrow{r} \arrow{d}{\pi_{n+1}} & \cdots \arrow{r} & L \arrow[dashed]{d}{\pi} \\ F_n \arrow{r}{\alpha_n} & F_{n+1} \arrow{r} & \cdots \arrow{r} & F \end{tikzcd} \]
Clearly $\pi$ is surjective.  Also any element in the kernel of $\pi$ must be a torsion element.  Since $A \otimes \mathcal{Q} \cong A$, we have $L \otimes \mathbb{Q} \cong L$, and, in particular, $L$ is torsion free.  So $\pi$ is an isomorphism.  Also, as each $\pi_n$  is positive and restricts surjective map on the positive cones, $\pi$ is an order isomorphism.  So $L \cong F$ as $\mathbb{Z}/2$-graded ordered groups.

Each $F_n$ is a direct sum of graded groups $\mathbb{Z} \oplus 0$ and $\mathbb{Z} \oplus \mathbb{Z}$ with the strict ordering in the first component.  Since $A$ has real rank zero and stable rank one, the ordered group $L = \mathrm{K}_*(A)$ has the Riesz Decomposition Property by Theorem 3.2 in \cite{ElliottATAlgebras}.  Appealing to Theorem 8.1 in \cite{ElliottATAlgebras}, there are C$^*$-algebras $B_n := C(\mathbb{T}) \otimes D_n$ with the $D_n$ finite dimensional and morphisms $B_n \rightarrow B_{n+1}$ such that $B = \underset{\longrightarrow}{\lim} \, B_n$ has real rank zero and $\mathrm{K}_*(B) \cong L$ as ordered groups.  By replacing $B$ with $B \otimes \mathbb{K}$, we may assume $B$ is stable so that $\mathcal{D}(B) = \mathrm{K}_*^+(B)$.  Since
\[ \mathrm{K}_*(B \otimes \mathcal{Q}) = \mathrm{K}_*(B) \otimes \mathbb{Q} = \mathrm{K}_*(A) \otimes \mathbb{Q} = \mathrm{K}_*(A \otimes \mathcal{Q}) = \mathrm{K}_*(A) = \mathrm{K}_*(B), \]
$B \cong B \otimes \mathcal{Q}$ by Elliott's Classification Theorem for A$\mathbb{T}$-algebras of real rank zero (Theorem \ref{thm:ElliottATAlgebras} above).  Hence
\[ (\underline{\mathrm{K}}(B), \underline{\mathrm{K}}(B)^+) \cong (\mathrm{K}_*(B), \mathrm{K}_*(B)^+). \]
Also, $B$ is an ASH-algebra with slow dimension growth and we have
\[ (\underline{\mathrm{K}}(A), \underline{\mathrm{K}}(A)^+) \cong (\underline{\mathrm{K}}(B), \underline{\mathrm{K}}(B)^+). \]
Therefore, $A \otimes \mathbb{K} \cong B \otimes \mathbb{K} \cong B$ by Dadarlat and Gong's classification theorem (Theorem 9.1 in \cite{DadarlatGong}).  So $A \otimes \mathbb{K}$ is an A$\mathbb{T}$-algebra.  Since corners of A$\mathbb{T}$-algebras are A$\mathbb{T}$, $A$ is an A$\mathbb{T}$-algebra.  This completes the proof of Theorem \ref{thm:ASHimpliesRationallyAT}.
\end{proof}

\section{Proofs of the Main Results}\label{sec:Proofs}

In this section we prove the results stated in the introduction.  For the tracial version of Theorem 1.1, the following proposition of de la Harpe and Skandalis is needed (see \cite{HarpeSkandalis}).

\begin{proposition}\label{prop:DixmierPowersProperty}
Let $A$ be a C*-algebra and consider an action $\alpha : \mathbb{F}_r \rightarrow\Aut(A)$ with $r \geq 2$.  If $\tau$ is a trace on $A \rtimes_{\lambda} \mathbb{F}_r$, there is a trace $\tau_0$ on $A$ such that $\tau = \tau_0 \circ \mathbb{E}$.
\end{proposition}

\begin{proof}[Proof of Theorem \ref{thm:MainResult}]
It is well known that (1) implies (2), and (2) implies (3) is Proposition~\ref{prop:StablyFiniteImpliesCoboundary}.

Assuming $\alpha$ satisfies (3), the induced action $\tilde{\alpha}$ of $\mathbb{F}_r$ on the unitization $\tilde{A}$ of $A$ also satisfies (3).  Moreover, if $\tilde{A} \rtimes_\lambda \mathbb{F}_r$ is MF, then $A \rtimes_\lambda \mathbb{F}_r$ is MF.  Hence to verify (3) implies (1), we may assume $A$ is unital.  Also, we may assume $r < \infty$.  Similar considerations show the claim about traces can be reduced to the case when $A$ is unital and $r < \infty$.

Note that $A \otimes \mathcal{Q}$ is an A$\mathbb{T}$-algebra by Theorem \ref{thm:ASHimpliesRationallyAT}.  Moreover, $\alpha$ induces an action $\alpha \otimes \operatorname{id} : \mathbb{F}_r \curvearrowright A \otimes \mathcal{Q}$ and traces on $A \rtimes_\lambda \mathbb{F}_r$ are in bijection with traces on $(A \otimes \mathcal{Q}) \rtimes_\lambda \mathbb{F}_r$.  If a trace on $(A \otimes \mathcal{Q}) \rtimes_\lambda \mathbb{F}_r$ is MF, then its restriction to $A \rtimes_\lambda \mathbb{F}_r$ is MF.  Hence the claim about traces can be reduced to the case where $A$ is an A$\mathbb{T}$-algebra of real rank zero.

The implication (3) implies (1) can also be reduced to the case when $A$ is A$\mathbb{T}$ of real rank zero.  Indeed, if $(A \otimes \mathcal{Q}) \rtimes_\lambda \mathbb{F}_r$ is MF, then $A \rtimes_\lambda \mathbb{F}_r$ is MF.  So to complete the reduction, we must show if $\alpha$ satisfies (3), so does $\alpha \otimes \operatorname{id}$.  To this end, note that
\[ \mathrm{K}_0(A \otimes \mathcal{Q}) \cong \mathrm{K}_0(A) \otimes \mathbb{Q} \cong \left\{ \frac{x}{n} : x \in \mathrm{K}_0(A), n \in \mathbb{N} \right\} \]
where we identify $x/m$ with $y/n$ if $k(nx - my) = 0$ for some $k \geq 1$ and the addition is defined by the usual arithmetic of fractions.
The morphism $\mathrm{K}_0(A) \rightarrow \mathrm{K}_0(A \otimes \mathcal{Q})$ induced by the embedding $A \hookrightarrow A \otimes \mathcal{Q}$ is given by $x \mapsto x / 1$ and the automorphisms $\mathrm{K}_0(\alpha_s) \otimes \operatorname{id}$ are given by $x / n \mapsto \mathrm{K}_0(\alpha_s)(x) / n$.
Also, under this identification,
\[ \mathrm{K}_0(A \otimes \mathcal{Q})^+ = \left\{ \frac{x}{n} : x \in \mathrm{K}_0(A)^+, n \in \mathbb{N} \right\}, \]
and $x / n \geq 0$ in $\mathrm{K}_0(A \otimes \mathcal{Q})$ if and only if $kx \geq 0$ in $\mathrm{K}_0(A)$ for some $k \in \mathbb{N}$.

Now, suppose $x_1 / n_1, \ldots, x_k / n_k \in \mathrm{K}_0(A \otimes \mathcal{Q})$ and $s_1, \ldots, s_k \in \mathbb{F}_r$ with
\[ x := \sum_{i=1}^k \left( \frac{x_i}{n_i} - \mathrm{K}_0(\alpha_{s_i} \otimes \operatorname{id})\left(\frac{x_i}{n_i}\right) \right) \geq 0. \]
Letting $d \geq 1$ be a common multiple of $n_1, \ldots, n_k$ there is an integer $k \geq 1$ with
\[ x' := \sum_{i=1}^k \left( \frac{kd}{n_i} x_i - \mathrm{K}_0(\alpha_{s_i})\left(\frac{kd}{n_i} x_i\right) \right) \geq 0. \]
Hence $x' = 0$ since $\alpha$ satisfies the coboundary condition, and so $x = x' / dk = 0$.  This shows $\alpha \otimes \operatorname{id}$ satisfies the coboundary condition.

We have reduced the theorem to the case where $A$ is a unital A$\mathbb{T}$-algebra of real rank zero and $r < \infty$.  We work to reduce further to the case of AF-algebras.  By Corollary \ref{cor:K0EmeddingOfATAlgebras} there is an AF-algebra $B$ and an embedding $\varphi: A \rightarrow B$ such that $\mathrm{K}_0(\varphi)$ is an isomorphism.  Let $s_1, \ldots, s_r$ denote the generators of $\mathbb{F}_r$ and denote $\alpha_i := \alpha_{s_i}$.  By Elliott's classification of AF-algebras, the automorphism $\mathrm{K}_0(\varphi) \mathrm{K}_0(\alpha_i) \mathrm{K}_0(\varphi)^{-1}$ of $\mathrm{K}_0(B)$ lifts to an automorphism $\mathrm{K}_0(\beta'_i)$ of $A$ for each $i = 1, \ldots, r$.

Fix a finite set $F \subseteq B$ and $\varepsilon > 0$.  Since $\mathrm{K}_0(\beta_i' \circ \varphi) = \mathrm{K}_0(\varphi \circ \alpha_i)$, the uniqueness portion of Theorem \ref{thm:ElliottATAlgebras} implies for each $i = 1, \ldots, r$, there is a unitary $u_i \in B$ such that
\[ \| u_i \beta_i'(\varphi(a)) u_i^* - \varphi(\alpha_i(a)) \| < \varepsilon \qquad \text{for all $a \in F$}. \]
Define $\beta_i = \operatorname{ad}(u_i) \circ \beta'_i$ for $i = 1, \ldots, r$.  The $\beta_i$ induce an action $\beta : \mathbb{F}_r \curvearrowright B$ such that $\beta_{s_i} = \beta_i$ for each $i = 1, \ldots, r$.

If $\alpha$ satisfies the coboundary condition, then $\beta$ satisfies the coboundary condition since $\mathrm{K}_0(\varphi) : \mathrm{K}_0(A) \rightarrow \mathrm{K}_0(B)$ is an isomorphism intertwining the actions $\mathrm{K}_0(\alpha)$ and $\mathrm{K}_0(\beta)$.  Theorem \ref{thm:K0EquivariantQDQ} implies the existence faithful morphism $\mu : \mathrm{K}_0(B) \rightarrow \mathrm{K}_0(\mathcal{Q}_\omega)$ such that $\mu \circ \mathrm{K}_0(\beta_s) = \mu$ for every $s \in \mathbb{F}_r$.  Since $B$ is AF, there is a $\ast$-homomorphism $\psi : A \rightarrow \mathcal{Q}_\omega$ such that $\mathrm{K}_0(\psi) = \mu$ by Theorem \ref{thm:LiftingK0Approximations}.  As $\mathrm{K}_0(\psi) = \mu$ is faithful, $\ker(\psi)$ does not contain any non-zero projections and since $\ker(\psi)$ is AF, $\ker(\psi) = 0$. Hence $\psi$ is injective.  Also, since
\[ \mathrm{K}_0(\psi \circ \beta_s) = \mu \circ \mathrm{K}_0(\beta_s) = \mu = \mathrm{K}_0(\psi), \]
the morphisms $\psi \circ \beta_s$ and $\psi$ are unitarily equivalent by Theorem \ref{thm:LiftingK0Approximations}.  Since $\mathbb{F}_r$ is free, there is a unitary representation $w : \mathbb{F}_r \rightarrow \mathcal{U}(\mathcal{Q}_\omega)$ such that $\operatorname{ad}(w_s) \circ \psi = \psi \circ \beta_s$ for every $s \in \mathbb{F}_r$.  Now, $\psi \circ \varphi : A \rightarrow \mathcal{Q}_\omega$ is injective and
\[ \| \psi(\varphi(\alpha_{s_i}(a))) - w_{s_i} \psi(\varphi(a)) w_{s_i}^* \| < \varepsilon \]
for $a \in F$ and $i = 1, \ldots, r$.

We have shown this existence of a faithful approximate covariant representation of the dynamical system $(A, \mathbb{F}_r, \alpha)$ on $\mathcal{Q}_\omega$.  A reindexing argument shows there is a faithful covariant representation of $(A, \mathbb{F}_r, \alpha)$ on $\mathcal{Q}_\omega$ and hence $\alpha$ is MF.  Now, Corollary \ref{cor:FreeGroupCrossedProducts} implies $A \rtimes_\lambda \mathbb{F}_r$ is MF, proving (1).

Finally, we must show every trace on $A \rtimes_\lambda \mathbb{F}_r$ is MF for an arbitrary action $\alpha : \mathbb{F}_r \curvearrowright A$.  We first show every invariant trace $\tau_A$ on $A$ is $\alpha$-MF.  The proof is very similar to the proof of (3) implies (1) above, so we omit some details.  As above, there is an AF-algebra $B$, an action $\beta : \mathbb{F}_r \curvearrowright B$, and an embedding $\varphi : A \rightarrow B$ approximately intertwining $\alpha$ and $\beta$ such that $\mathrm{K}_0(\varphi)$ is an isomorphism.  The state $\mathrm{K}_0(\tau_A) \circ \mathrm{K}_0(\varphi)^{-1}$ on $\mathrm{K}_0(B)$ lifts to a trace $\tau_B$ on $B$.  Since $\mathrm{K}_0(\tau_A) = \mathrm{K}_0(\tau_B \circ \varphi)$ on $\mathrm{K}_0(A)$ and $A$ has real rank zero, $\tau_A = \tau_B \circ \varphi$.

Note that $\mathrm{K}_0(\tau_B)$ is a $\mathrm{K}_0(\beta)$-invariant state on $\mathrm{K}_0(B)$.  By Theorem \ref{thm:K0QDTQ}, there is a state-preserving morphism $\mathrm{K}_0(B) \rightarrow \mathrm{K}_0(\mathcal{Q}_\omega)$ which is constant on $\mathrm{K}_0(\beta)$-orbits.  As above, Theorem \ref{thm:LiftingK0Approximations} allows one to lift this to a (necessarily trace-preserving) covariant representation of $(B, \mathbb{F}_r, \beta)$ on $\mathcal{Q}_\omega$.  Then, after composing with $\varphi : A \rightarrow B$ and applying an ultrafilter reindexing argument, one obtains a trace-preserving covariant representation of $(A, \mathbb{F}_r, \alpha)$ on $\mathcal{Q}_\omega$.  Thus $\tau_A$ is $\alpha$-MF.

Having shown every $\alpha$-invariant trace on $A$ is $\alpha$-MF, it follows that every trace on $A \rtimes_\lambda \mathbb{F}_r$ of the form $\tau \circ \mathbb{E}$ for an invariant trace $\tau$ on $A$ is MF by Corollary \ref{cor:FreeGroupCrossedProducts}.  In particular, when $r \geq 2$, every trace on $A \rtimes_\lambda \mathbb{F}_r$ is MF by Proposition \ref{prop:DixmierPowersProperty}.

Now consider the case when $r = 1$.  Let $\tau$ be a trace on $A \rtimes_\alpha \mathbb{Z}$ and let
\[ J = \{ a \in A : \tau(a^*a) = 0 \}. \]
Then $J$ is an $\alpha$-invariant ideal of $A$ and $\tau$ vanishes on the ideal $J \rtimes_\alpha \mathbb{Z}$ of $A \rtimes_\alpha \mathbb{Z}$.  Let $\bar{\alpha}$ denote the action on $A / J$ induced by $\alpha$ and let $\bar{\tau}$ denote the induced trace on $(A / J) \rtimes_{\bar{\alpha}} \mathbb{Z}$.  Then $\bar{\tau}$ is a $\bar{\alpha}$-invariant trace on $A / J$.

The quotient $A / J$ is an A$\mathbb{T}$-algebra and hence $(A / J) \rtimes_{\bar{\alpha}} \mathbb{Z}$ is a separable, nuclear C$^*$-algebra satisfying the UCT. Also, $(A / J) \rtimes_{\bar{\alpha}} \mathbb{Z}$ admits a faithful trace; namely, the trace $\bar{\tau}|_{A / J} \circ \mathbb{E}$ is faithful. By Corollary 6.1 in \cite{TikuisisWhiteWinter}, every trace on $(A / J) \rtimes_{\bar{\alpha}} \mathbb{Z}$ is quasidiagonal.  In particular, $\bar{\tau}$ is quasidiagonal and hence $\tau$ is quasidiagonal.  This completes the proof.
\end{proof}

\begin{proof}[Proof of Theorem \ref{thm:ByTheMegaTheorem...}]
We have (1) implies (2) since every MF algebra is stably finite.  Since $A \rtimes_\lambda \mathbb{F}_r$ is exact, if $A \rtimes_\lambda \mathbb{F}_r$ is stably finite, then $A \rtimes_\lambda \mathbb{F}_r$ admits a trace by Corollary 5.12 in \cite{HaagerupQuasitrace}.

Note that every invariant trace on $A$ induces a faithful invariant trace on $A \rtimes_\lambda \mathbb{F}_r$.  Hence the proof will be complete once we show every trace on $A \rtimes_\lambda \mathbb{F}_r$ is MF.  By assumption, $A \otimes \mathcal{Q}$ is a separable, simple, unital C$^*$-algebra which satisfies the UCT and has finite nuclear dimension.  Assuming $A$ has a trace, Theorem 6.2(iii) in \cite{TikuisisWhiteWinter} implies $A \otimes \mathcal{Q}$ is an ASH-algebra.  Since $\mathrm{K}_0(A)$ separates traces on $A$, $A \otimes \mathcal{Q}$ has real rank zero by Theorem 1.4(f) in \cite{BlackadarKumjianRordam} and the paragraph that follows it.  Hence Theorem \ref{thm:MainResult} applies.
\end{proof}

\begin{proof}[Proof of Corollary \ref{cor:Monotracial}]
By the main result of \cite{TikuisisWhiteWinter}, $A$ is quasidiagonal and hence $A \otimes \mathcal{Q}$ has finite nuclear dimension by the main result of \cite{MatuiSato}.  The unique trace on $A$ is certainly invariant under $\alpha$ and the first part of the corollary follows from Theorem \ref{thm:ByTheMegaTheorem...}.

A well known theorem of Rosenberg (see the appendix of \cite{Rosenberg}) states that if C$^*_\lambda(G)$ is quasidiagonal, then $G$ is amenable.  Hence
if $r \geq 2$, $A \rtimes_\lambda \mathbb{F}_r$ is MF but not quasidiagonal and the result follows from Corollary 13.5 in \cite{BrownQDPaper}.
\end{proof}

\begin{proof}[Proof of Theorem \ref{thm:AmenableByFreeGroups}]
Following \cite{OzawaRordamSato}, we define the Bernoulli shift crossed product
\[ B(G) := \left( \bigotimes_G \mathbb{M}_{2^\infty} \right) \rtimes_\lambda G, \]
where the action is given by left translation on the index set.  By Proposition 2.1 in \cite{OzawaRordamSato}, $B(G)$ is a separable, simple, unital C$^*$-algebra with a unique trace and satisfies the UCT.  The action $\mathbb{F}_r \curvearrowright G$ induces an action $\mathbb{F}_r \curvearrowright B(G)$ such that the embedding C$^*_\lambda(G) \hookrightarrow B(G)$ is equivariant.  Hence we have and embedding
\[ C^*_\lambda(G \rtimes \mathbb{F}_r) \cong C^*_\lambda(G) \rtimes_\lambda \mathbb{F}_r \hookrightarrow B(G) \rtimes_\lambda \mathbb{F}_r. \]
By Corollary \ref{cor:Monotracial}, $B(G) \rtimes_\lambda \mathbb{F}_r$ is MF and hence C$^*_\lambda(G \rtimes \mathbb{F}_r)$ is MF.  Since the embedding above is trace-preserving, we also have that the canonical trace on C$^*_\lambda(G \rtimes \mathbb{F}_r)$ is MF.
\end{proof}

\section{A Result for Cuntz-Pimnser Algebras}

In \cite{Schafhauser}, the second author has shown how to reduce certain questions about Cuntz-Pimsner algebras to questions about crossed products by $\mathbb{Z}$.  Using these techniques, it is possible to prove a version of Theorem \ref{thm:MainResult} for Cuntz-Pimsner algebras with coefficients in a real rank zero ASH-algebra (see Theorem \ref{thm:CuntzPimsner} below).  The reader is referred to Katsura's paper \cite{Katsura:Corr1} for the relevant background on C$^*$-correspondences and Cuntz-Pimsner algebras.

Given a C$^*$-correspondence $H$ over a C$^*$-algebra $A$, taking the external tensor product with $\mathcal{Q}$ yields a correspondence $H \otimes \mathcal{Q}$ over $A \otimes \mathcal{Q}$.  Let $(\pi^0, \pi^1)$ and $(\tilde{\pi}^0, \tilde{\pi}^1)$ denote the universal covariant representations of $(A, H)$ and $(A \otimes \mathcal{Q}, H \otimes \mathcal{Q})$ on the Cuntz-Pimsner algebras $\mathcal{O}_A(H)$ and $\mathcal{O}_{A \otimes \mathcal{Q}}(H \otimes \mathcal{Q})$, respectively.

\begin{lemma}\label{lem:RationalCorrespondence}
For any C*-correspondence $H$ over a C*-algebra $A$, there is a natural isomorphism
\[ \varphi : \mathcal{O}_A(H) \otimes \mathcal{Q} \rightarrow \mathcal{O}_{A \otimes \mathcal{Q}}(H \otimes \mathcal{Q}) \]
such that $\varphi(\pi^0(a) \otimes b) = \tilde{\pi}^0(a \otimes b)$ and $\varphi(\pi^1(\xi) \otimes b) = \tilde{\pi}^1(\xi \otimes b)$ for each $a \in A$, $\xi \in H$, and $b \in \mathcal{Q}$.
\end{lemma}

\begin{proof}
The maps
\[ \rho^0 : A \otimes \mathcal{Q} \rightarrow \mathcal{O}_A(H) \otimes \mathcal{Q} \qquad \text{and} \qquad \rho^1 : H \otimes \mathcal{Q} \rightarrow \mathcal{O}_A(H) \otimes Q \]
given by $\rho^0(a \otimes b) = \pi^0(a) \otimes b$ and $\rho^1(\xi \otimes b) = \pi^1(\xi) \otimes b$ form a faithful Toeplitz representation which admits a gauge action.  Moreover, $\mathcal{O}_A(H) \otimes Q$ is generated as a C$^*$-algebra by $\rho^0(A \otimes \mathcal{Q})$ and the $\rho^1(H \otimes \mathcal{Q})$.  Hence there is an induced surjective $*$-homomorphism
\[ \varphi : \mathcal{O}_A(H) \otimes \mathcal{Q} \rightarrow \mathcal{O}_{A \otimes \mathcal{Q}}(H \otimes \mathcal{Q}) \]
such that $\varphi \circ \tilde{\pi}^0 = \rho^0$ and $\varphi \circ \tilde{\pi}^1 = \rho_1$ by Proposition 7.14 in \cite{Katsura:Corr2}.

It's enough to show $\varphi$ is injective.  Assume $a \in \mathcal{O}_A(H)$ and $b \in \mathcal{Q}$ such that $b \neq 0$ and $\varphi(a \otimes b) = 0$.  By Kirchberg's Slice Lemma \cite[Lemma 4.1.9]{RordamYellowBook}, it's enough to show $a = 0$.  Since $\mathcal{Q}$ is simple, there are $x_1, \ldots, x_n, y_1, \ldots, y_n \in \mathcal{Q}$ with $\sum_i x_i b y_i = 1$.  Now,
\[ \varphi(a \otimes 1) = \sum_i \varphi(1 \otimes x_i) \varphi(a \otimes b) \varphi(1 \otimes y_i) = 0. \]
By the Gauge Invariant Uniqueness Theorem, the restriction of $\varphi$ to $\mathcal{O}_A(H)$ is injective, and hence $a = 0$.
\end{proof}

\begin{definition}
A C$^*$-correspondence $H$ over a C$^*$-algebra $A$ is called \emph{regular} if the left action $A \rightarrow \mathbb{B}(H)$ is faithful and has range contained in $\mathbb{K}(H)$.
\end{definition}

Given a regular C$^*$-correspondence $(A, H)$, taking the Kasparov product with $H$ yields endomorphism $[H]$ of $\mathrm{K}_0(A)$; viewing $\mathrm{K}_0(A)$ as the Grothendieck group of finitely generated Hilbert $A$-modules, the endomorphism $[H]$ is determined by the assignment $L \mapsto L \otimes_A H$ for every finitely generated Hilbert $A$-module $L$.  If $(A, H)$ is a regular C$^*$-correspondence, then the C$^*$-correspondence $(A \otimes \mathcal{Q}, H \otimes \mathcal{Q})$ is also regular.  Making the identification
\[ \mathrm{K}_0(A \otimes \mathcal{Q}) \cong \left\{ \frac{x}{n} : x \in \mathrm{K}_0(A), n \in \mathbb{N} \right\}, \]
the endomorphism $[H \otimes \mathcal{Q}]$ is given by
\[ [H \otimes \mathcal{Q}]\left(\frac{x}{n}\right) = \frac{[H](x)}{n}. \]

\begin{theorem}\label{thm:CuntzPimsner}
Let $A$ be an ASH-algebra such that $A \otimes \mathcal{Q}$ has real rank zero.  If $H$ is a separable, regular C*-correspondence over $A$, then the following are equivalent:
\begin{enumerate}
  \item $\mathcal{O}_A(H)$ is quasidiagonal;
  \item $\mathcal{O}_A(H)$ is stably finite;
  \item if $x \in \mathrm{K}_0(A)$ with $[H](x) \leq x$, then $[H](x) = x$.
\end{enumerate}
\end{theorem}

\begin{proof}
Using the results above and arguing as in the proof of Theorem \ref{thm:MainResult} in Section \ref{sec:Proofs}, we may assume $A \otimes \mathcal{Q} \cong A$. Then $A$ is an A$\mathbb{T}$-algebra of real rank zero by Theorem \ref{thm:ASHimpliesRationallyAT}.

Let $\gamma$ denotes the gauge action of $\mathbb{T}$ on $\mathcal{O}_A(H)$ and define $B := \mathcal{O}_A(H) \rtimes \mathbb{T}$.  Combing Propositions 5.1, 5.2, and 5.3 in \cite{Schafhauser} shows $B$ is a direct limit of C$^*$-algebras Morita equivalent to $A$.  In particular, $B$ is an A$\mathbb{T}$-algebra of real rank zero as follows from the permanence properties given in Proposition 3.2.5 of \cite{RordamYellowBook}.  If $\hat{\gamma}$ denotes the dual action of $\mathbb{Z}$ on $B$, then $\mathcal{O}_A(H)$ is Morita equivalent to $B \rtimes_\gamma \mathbb{Z}$ by Takai duality.  The result can be deduced from Theorem \ref{thm:MainResult} as in the proof of Theorem C in \cite{Schafhauser}.  Note in particular that since $A$ is nuclear, $\mathcal{O}_A(H)$ is nuclear by Corollary 7.4 of \cite{Katsura:Corr1}; hence $\mathcal{O}_A(H)$ is quasidiagonal if and only if $\mathcal{O}_A(H)$ is MF by the Choi-Effros Lifting Theorem.
\end{proof}

\end{document}